\documentclass[a4paper,twoside,11pt,reqno]{amsart}
\usepackage[margin=1in]{geometry}
\usepackage[english]{babel}
\usepackage[utf8]{inputenc}
\usepackage{amsmath}
\usepackage{amssymb}
\usepackage{amsfonts}
\usepackage{amsthm}
\usepackage{bm}
\usepackage{mathrsfs}
\usepackage{hyperref}
\usepackage{comment}
\hypersetup{
	colorlinks=true,
	linkcolor=blue,
	filecolor=blue,
	citecolor=purple,      
	urlcolor=cyan,
}
\usepackage{xcolor}
\usepackage{dsfont}
\usepackage{physics}
\usepackage{stmaryrd}
\usepackage{enumitem}

\usepackage{tikz-cd}
\usepackage{tikz}
\usepackage{tikzsymbols}
\usepackage{mathtools}
\usetikzlibrary{matrix}

\usepackage[matrix,arrow,cmtip]{xy} % curve, frame, graph, 2cell, all
\SelectTips{cm}{}
\newdir{(}{{}*!/-5pt/@^{(}}
\newdir{(x}{{}*!/-5pt/@_{(}}
\newdir{+}{{}*!/-9pt/{}}
\newdir{>+}{@{>}*!/-9pt/{}}
\entrymodifiers={+!!<0pt,\fontdimen22\textfont2>}
\theoremstyle{plain}
\newtheorem{theorem}[subsubsection]{Theorem}

\newtheorem{lemma}[subsubsection]{Lemma}
\newtheorem{proposition}[subsubsection]{Proposition}
\newtheorem{notation}[subsubsection]{Notation}
\newtheorem{corollary}[subsubsection]{Corollary}
\theoremstyle{definition}
\newtheorem{definition}[subsubsection]{Definition}

\newtheorem{example}[subsubsection]{Example}
\theoremstyle{remark}
\newtheorem{remark}[subsubsection]{Remark}

\numberwithin{equation}{section}

\newcommand{\CC}{\mathbb{C}}
\newcommand{\RR}{\mathbb{R}}
\newcommand{\QQ}{\mathbb{Q}}
\newcommand{\ZZ}{\mathbb{Z}}
\newcommand{\OO}{\mathcal{O}}
\newcommand{\HH}{\mathrm{H}}

\newcommand{\uX}{\underline{X}}

\newcommand{\vv}{\mathbf{v}}
\newcommand{\uu}{\mathbf{u}}

\newcommand{\PP}{\mathbb{P}}

\newcommand{\Hom}{\mathrm{Hom}}
\newcommand{\Ext}{\mathrm{Ext}}
\newcommand{\CH}{\mathrm{CH}}
\newcommand{\Coh}{\mathrm{Coh}}
\newcommand{\ch}{\mathrm{ch}}

\newcommand{\Amp}{\mathrm{Amp}}
\newcommand{\NS}{\mathrm{NS}}
\newcommand{\gr}{\mathrm{gr}}

\newcommand{\cal}{\mathcal}

\newcommand{\Gm}{\mathbb{G}_{m}}

\newcommand{\LL}{\mathbf{L}}
\newcommand{\HN}{Harder-Narasimhan }

\newcommand{\Num}{\mathrm{Num}}
\newcommand{\beps}{\bm{\epsilon}}
\newcommand{\cA}{\mathcal{A}}

\newcommand{\sC}{\mathscr{C}}
\newcommand{\sX}{\mathscr{X}}
\newcommand{\cH}{\mathcal{H}}
\newcommand{\id}{\mathrm{id}}
\newcommand{\tor}{\mathrm{tor}}
\newcommand{\too}{\longrightarrow}

\newcommand{\leqp}{%
  \mathrel{\raisebox{-0.5ex}{$\scriptscriptstyle($}}%
  \leq
  \mathrel{\raisebox{-0.5ex}{$\scriptscriptstyle)$}}%
}

\begin{document}

		\title{Stability conditions on surface root stacks}
		\date{}
		\author{Yeqin Liu and Yu Shen}
  \address{Department of Mathematics, University of Michigan, 530 Church St,
Ann Arbor, MI 48109, USA}
\email{yqnl@umich.edu}
\address{Department of Mathematics, Michigan State University, 619 Red Cedar Road, East Lansing, MI 48824, USA}
  \email{shenyu5@msu.edu}

\begin{abstract}
    We construct tilt stability conditions on surface root stacks and show that they have support property with respect to the rational Chen-Ruan cohomology.
\end{abstract}

\maketitle

\setcounter{tocdepth}{1}
\tableofcontents

\section{Introduction}

Compared to projective varieties, the geometry of
Deligne-Mumford stacks are less understood.
From the perspective to study derived categories and moduli spaces of sheaves, Bridgeland stability conditions, first introduced in \cite{Dou2, Bri07}, is a powerful tool to study the underlying categories. 
Various successful applications have been made, such as studying classical moduli spaces of sheaves \cite{ABCH13, CHW17, LZ18, LZ19}, hyperK\"{a}hler geometry \cite{MS19, BLM+, LPZ22}, and several enumerative problems \cite{Tod12, BS16, LR22b}. 

In general, it is a central question in this area that whether there exist stability conditions with nice deformation properties on any given triangulated category. Several existence results have been obtained, such as projective curves \cite{Mac07}, surfaces \cite{Bri08, ABL23}, and some threefolds \cite{Mac14, MP15, BMS16}. However, in general the existence of Bridgeland stability conditions is a challenging problem. In this paper, we construct stability conditions on surface root stacks using tilt, and show that their deformation dimensions are at least the dimension of rational Chen-Ruan cohomology. Our main result is the following theorem.

\begin{theorem}[Theorem \ref{maintheorem}]\label{Main theorem}
Let $X$ be a smooth projective surface, $C\subset X$ be a smooth curve, and $\uX=\sqrt[n]{(X, C)}$ be the $n$-th root stack ramified along the curve (Definition \ref{rootstack}). 

 The tilt stability conditions (Definition \ref{tilt stability}) on $\uX$ are Bridgeland stability conditions. They have the support property with respect to the rational Chen-Ruan cohomology.    
\end{theorem}

There has been some study for stability conditions on stacks. In \cite{Rot19}, Rota studies stability conditions on 2-root curves. Later, Lim and Rota study stability conditions on the canonical stack associated with a projecitve surface with ADE singularities \cite{LR22}, unifying the work \cite{Bri09}.
It is worth mentioning that in \cite{CP10}, the authors constructed stability conditions on certain root stacks such as 2-root stacks over curves, $\PP^{2}$ and $ \PP^{3}$. Their stability conditions are glued from semiorthogonal components, and the stability conditions on projective spaces are Euler \cite{Mac04}, which is different from our approach. 

Our theorem is also useful in studying parabolic sheaves on surfaces \cite{BV12} and their moduli spaces, since the modified Gieseker stability (\cite{Nir08}) is naturally related to tilt stability by the following theorem known as the large volume limit. 

\begin{theorem}[Large volume limit, Theorem \ref{largevolumelimit}]
   Let $B, H \in \NS(X)_{\RR}$ with $H$ ample, $ \sigma_{B, tH}:=(Z_{B,tH}, \cA_{B ,tH}), t>0$. Suppose $E\in D^{b}(\uX)$ satisfies $r(E), \mu_{H}^{B}(E)> 0$. Then $E$ is semistable under $\sigma_{B, tH}$ for all $t\gg 0$ precisely if $E$ is a shift of a $(B,H)$-twisted semistable sheaf on $\uX$.
\end{theorem}

We would like to explain the main point of this paper.
Thanks to the Bogomolov Inequality on Deligne-Mumford surfaces \cite{JK24}, it is easily seen that the tilt stability conditions are pre-stability conditions, in the sense that nonzero stable objects have nonzero central charges (Definition \ref{pre-stab}). They also have the support property with respect to the \emph{usual cohomology}. However, the difficulty is to prove the support property with respect to the \emph{Chen-Ruan cohomology}, which is a larger lattice that reflects the numerical class of sheaves on stacky locus. 

The Chen-Ruan cohomology was first introduced in \cite{CR04} to study Gromov-Witten theory \cite{CR02}. As a refinement of the usual cohomology, there is an associative algebra structure on the Chen-Ruan cohomology \cite{AGV01}, as a generalization of intersection theory on stacks \cite{Vis89}.
Note that the Riemann-Roch for Deligne-Mumford stacks \cite{Toe99} necessarily involves Chern classes on the first inertia stacks. The numerical invariants of sheaves on a Deligne-Mumford stack naturally live in the Chen-Ruan cohomology, since the usual cohomology loses information on the stacky locus. 
To better emphasize this point, we also note that the rational $K$ group of a smooth Deligne-Mumford stack is isomorphic to the rational Chen-Ruan Chow group  \cite{ALR07}.

 We would like to mention our approach to prove the support property.
  First, we prove the support property on the ordinary cohomology (defined in \ref{lattice}). Then, we delicately construct subobjects of every stable object $E$ that reflect its orbifold Chern class. By stability of $E$, the presence of these subobjects bounds the orbifold Chern class of $E$. We illustrate the central ideas in Example \ref{exampleP} and Example \ref{exampleQ}. It is known that the support property on a lattice is equivalent to the existence of certain quadratic forms (Definition \ref{quadratic}). However in our case, it appears that finding such quadratic forms is very difficult.

We expect that the glued stability conditions from the semiorthogonal decomposition $D^{b}(\uX)=\langle D^{b}(C)\rho_{1},...,D^{b}(C)\rho_{n-1}, D^{b}(X) \rangle$ also have the support property with respect to the rational Chen-Ruan cohomology. Using the techniques in \cite{CP10}, this is already true for some special cases, such as when $X$ admits a full exceptional collection or $C\cong \PP^{1}$. It is an interesting question that whether the two different types of stability conditions lie in the same connected component of $\mathrm{Stab}(\uX)$.

\subsection{Outline of Paper} 

In section \ref{section2}, we recall the preliminaries of root stacks, their Chen-Ruan cohomology, and Bridgeland stability.
In section \ref{section3}, we recall the structure of $D^{b}(\uX)$ and compute Chern class formulas for coherent sheaves on $\uX$.
In section \ref{section4}, we construct the tilt stability functions and show that they are pre-stability conditions (Definition \ref{tilt stability}).
In section \ref{section5}, we show the tilt stability conditions have the support property with respect to the rational Chen-Ruan cohomology. 

\subsection{Acknowledgments} We thank Izzet Coskun, Andres Fernandez Herrero, Rajesh Kulkarni, Alexander Perry and Nick Rekuski for many helpful discussions and comments.

The second author was partially support by NSF grant DMS-2101761.
\subsection{Notation and conventions} All varieties are smooth projective over $\CC$.

Fix any $n\in \mathbb{Z}_{>0}$. Let $X$ be a smooth projective surface and $C\subset X$ be a smooth curve. Let $\uX=\sqrt[n]{(X, C)}$ be the $n$-th root stack along the curve (Definition \ref{rootstack}), and $\sX=[\OO_{X}(C)^{\times}/ (\Gm, n)]$ be the $\mu_{n}$-gerbe over $X$ (Definition \ref{rootgerbe}). 

Let $\mu_{m}=\ZZ/m \ZZ$ be the cyclic group of order $m$. For every group $G$, let $\mathbf{B}G=[*/G]$ be its classifying stack.

\section{Preliminaries}\label{section2}

\subsection{Root gerbe and root stack}

In this subsection we recall the definitions and properties of root gerbes and root stacks. 

\subsubsection{Constructions}\label{gerbe} 
First we recall the construction of root gerbes (see e.g. \cite{Alp23}).
\begin{definition}\label{rootgerbe}
Fix $n\in \ZZ_{>0}$.
Let $X$ be a scheme and $L$ be a line bundle on $X$, which has the classifying morphism $[L]: X\to \mathbf{B}\mathbb{G}_{m}$. Let $n:\mathbf{B}\mathbb{G}_{m}\to \mathbf{B}\mathbb{G}_{m}$ be the morphism induced from the $n$-th power map $\mathbb{G}_{m}\to \mathbb{G}_{m}: t\to t^n$. Define the $n$-th root gerbe $\mathscr{X}$ to be the fibered product 

\[ \begin{tikzcd}
\mathscr{X} \arrow{r} \arrow[swap]{d} {p}&\mathbf{B}\mathbb{G}_{m} \arrow{d}{n} \\%
X \arrow{r}{[L]}& \mathbf{B}\mathbb{G}_{m}.
\end{tikzcd}
\]
\end{definition}
Next we recall the construction of root stacks.
\begin{definition}\label{rootstack}
Using the notations in Definition \ref{rootgerbe}, let $s\in \Gamma(X, L)$ be a section. This data determines a morphism $[L,s]: X\to [\mathbb{A}^{1}/\mathbb{G}_{m}]$. Let $n: [\mathbb{A}^{1}/\mathbb{G}_{m}]\to [\mathbb{A}^{1}/\mathbb{G}_{m}] $ induced from the $n$th power map $n: \mathbb{G}_{m}\to \mathbb{G}_{m}$, given by $x\to x^{n}$, which is equivariant under the $n$-th power map $n:\mathbb{G}_{m}\to \mathbb{G}_{m}$. Define the $n$-th root stack $\uX:=\sqrt[n]{(X, C)}$ of $X$ ramified along $C$ to be the fiber product
\[ \begin{tikzcd}
\sqrt[n]{(X, C)} \arrow{r} \arrow[swap]{d}{\pi} & \lbrack \mathbb{A}^{1}/\mathbb{G}_{m} \rbrack
\arrow{d}{n} \\%
X \arrow{r}{[L,s]}& \lbrack \mathbb{A}^{1}/\mathbb{G}_{m} \rbrack.
\end{tikzcd}
\]
\end{definition}
Universal line bundles on root gerbes and root stacks are defined as follows.
\begin{definition}\label{universal}
By \cite{IU15}, there exist universal line bundles $\underline{\mathcal{M}}\in \Coh(\uX)$ and $\mathcal{M}\in \Coh(\sX)$. They satisfy the following properties:
\begin{itemize}
\item
We have $\underline{\mathcal{M}}^{\otimes n}\cong \pi^{*}L$. Let $\underline{\rho_{k}}:=\underline{\mathcal{M}}^{-k}$ for $0\leq k\leq n-1$. 
\item 
We have $\mathcal{M}^{\otimes n}\cong p^{*}L$. Let $\rho_{k}:=\mathcal{M}^{-k}$ for $0\leq k \leq n-1$. 
\end{itemize}
\end{definition}

\subsubsection{Properties}

We first collect some properties of root gerbes (e.g. \cite[Chatper 3.9.2]{Alp23}).

\begin{lemma}[{\cite[Theorem 1.5]{IU15}}]\label{decomposition}
  The category $\Coh(\sX)$ splits as the following direct sum:
$$\operatorname{Coh}(\mathscr{X})\cong \operatorname{Coh}(X)\rho_{0} \oplus \operatorname{Coh}(X)\rho_{1}...\oplus \operatorname{Coh}(X)\rho_{n-1}.$$
\end{lemma}

\begin{proposition}\label{rootgerbeDM}
The root gerbe $\sX$ in Definition \ref{rootgerbe} has following properties:
\begin{enumerate}
    \item $\sX$ is a Deligne-Mumford stack.
    \item The fiber of $p: \sX \to X$ at a closed point $x\in X$ is isomorphic to $ \mathbf{B}\mu_{n}$.
    \item If $L=\mathcal{O}_{X}$ is trivial in Definition \ref{gerbe}, then $$\sX \cong [X /\mu_{n}]\cong X \times \mathbf{B}\mu_{n}, $$
where $\mu_{n}$ acts trivially on $X$.
\end{enumerate}
\end{proposition}    

% For simplicity, we will write $\mathscr{X}$ in place of  $\mathscr{X}$ when there is no confusion.

Next we collect properties of root stacks (see e.g. \cite[Chatper 3.9.2]{Alp23}).
Let $\bar{j}: C\hookrightarrow X$ be a divisor. We have the following commutative diagram
\begin{equation}\label{square}
 \begin{tikzcd}
\mathscr{C}\arrow{r}{j} \arrow[swap]{d}{p} &\uX \arrow{d}{\pi} \\%
C \arrow{r}{\bar{j}}& X,
\end{tikzcd}
\end{equation}
where $\mathscr{C}$ is the $n$-th root gerbe over $C$ constructed by the line bundle $\mathcal{O}_{X}(C)|_{C}$.

\begin{lemma}[{\cite[Theorem 1.2]{BD23}}]\label{semiorthgonal} We have semiorthgonal decomposition as follow

$$D^{b}(\uX)=\langle D^{b}(C)\underline{\rho_{1}},...,D^{b}(C)\underline{\rho_{n-1}}, D^{b}(X) \rangle,$$

where the embeddings of $D^{b}(X)$ and $D^{b}(C)$ are respectively $\pi^{*}$ and $j_{*}p^{*}$.
\end{lemma}

\begin{proposition}\label{rootstackDM}
The root stack $\uX$ in Definition \ref{rootstack} has the following properties:
\begin{enumerate}
\item $\uX$ is a Deligne-Mumford stack.
\item  $\pi: \uX \to X$ is an isomorphism away from $C$.
\item The fiber of $\pi: \uX  \to X$ over a closed point $x\in C$ is isomorphic to $ \mathbf{B}\mu_{n}$.
\item If $X=\operatorname{Spec}(A)$ is an affine scheme, and $C$ is defined by an element $a\in A$ in the construction, then $$\uX \cong [\operatorname{Spec}(A[x]/(x^{n}-a))/\mu_{n}], $$
where $\mu_{n}$ acts on $\operatorname{Spec}(A[x]/(x^{n}-a))$ via $t\cdot x=tx$.
\end{enumerate}
\end{proposition}

%We will write $\uX$ for $\uX$ if there is no confusion.

\subsection{Orbifold cohomology}

In this subsection, we recall the notion of Chen-Ruan cohomology, introduced by \cite{CR04}.

First we collect some properties of inertia stacks (see e.g. \cite[Chapter 6.1]{Liu11}). Let $\mathcal{X}$ be a smooth Deligne-Mumford stack, the inertia stack $\mathcal{IX}$ associated to $\mathcal{X}$ is a smooth Deligne-Mumford stack such that the following diagram is Cartesian:

\begin{center}
\begin{tikzcd}
\mathcal{IX} \arrow[r] \arrow[d] & \mathcal{X} \arrow[d, "\Delta"] \\
\mathcal{X} \arrow[r, "\Delta"]      & \mathcal{X}\times \mathcal{X},               
\end{tikzcd}

\end{center}
where $\Delta: \mathcal{X}\to \mathcal{X}\times \mathcal{X}$ is the diagonal map. The objects in the category $\mathcal{IX}$ are 
$$\operatorname{Ob}(\mathcal{IX})=\{(x,g)|x\in \mathcal{X}, g\in \operatorname{Aut}_{\mathcal{X}}(x)  \}.$$
The morphisms between two objects in the category $\mathcal{IX}$ are
$$\operatorname{Hom}_{\mathcal{IX}}((x_{1},g_{1}), (x_{2},g_{2}))=\{h\in \operatorname{Hom}_{\mathcal{X}}(x_{1},x_{2})| h\circ g_{1}=g_{2}\circ h  \}.$$
In particular, we have
$$\operatorname{Aut}_{\mathcal{IX}}(x,g)=\{h\in \operatorname{Aut}_{\mathcal{X}}(x)| h\circ g=g\circ h  \}.$$
Assume $\mathcal{X}$ is connected, let $$\mathcal{IX}=\bigsqcup_{i\in I}\mathcal{X}_{i}$$ be the disjoint union of connected components. There is a distinguished connected component $\mathcal{X}_{0}$ called the \emph{untwisted sector}, whose objects are $(x, \operatorname{id}_{x})$, where $x\in \operatorname{Ob}(\mathcal{X})$, and $\operatorname{id}_{x}\in \operatorname{Aut}(x)$ is the identity element. The other connected components are called \emph{twisted sectors}.

\begin{example}
 Let $\uX$ be the root stack in Definition \ref{rootstack}, then $$\operatorname{Ob}(\mathcal{I}\uX)=\{(x, \operatorname{id}_{x} )|x\in \operatorname{Ob}(\mathcal{X})\}\bigsqcup\{(x, \sigma )|x\in \mathscr{C}\}\bigsqcup...\bigsqcup\{(x, \sigma^{n-1}) |x\in \mathscr{C} \},$$  
where $\sigma$ is a generator of the $n$-th cyclic group $\mu_{n}.$ Since
$\operatorname{Aut}_{\mathcal{I}\uX}(x,g)=\operatorname{Aut}_{\mathcal{X}}(x),$
we have $$\mathcal{I}\uX=\uX\bigsqcup\bigsqcup_{1\leq k\leq n-1}\mathscr{C}.$$
\end{example}

 Chen and Ruan give the structure of a graded ring of orbifold cohomology \cite{CR04}. However, in this paper, we only consider its additive structure.

\begin{definition}[{\cite[Definition 3.2.3]{CR04}}]
    The \emph{Chen-Ruan orbifold cohomology} of a smooth Deligne-Mumford Stack $\mathcal{X}$ is defined as 
    $$\HH^{*}_{CR}(\mathcal{X}):=\bigoplus_{i}\HH^{*}(\mathcal{X}_{i})=\HH^{*}(\mathcal{IX}).$$
   
\end{definition}

\begin{lemma}[{\cite[Proposition 36]{Beh04}}]\label{DMrational}
    Let $\mathcal{X}$ be a Deligne-Mumford stack with coarse space $X$, then the canonical morphism $\mathcal{X}\to X$ induces isomorphisms on $\mathbb{Q}$-valued cohomology groups:$$\HH^{*}(X,\mathbb{Q})\xrightarrow{\sim} \HH^{*}(\mathcal{X}, \mathbb{Q}). $$
\end{lemma}
Since the coarse spaces of $\uX$ and $\mathscr{C}$ are $X$ and $C$ respectively, the rational Chen-Ruan cohomology of $\uX$ is computed by the following lemma. 
\begin{lemma}
$$
\HH^{*}_{CR}(\uX,\mathbb{Q})
=\HH^{*}(\uX,\mathbb{Q}) \oplus\bigoplus_{1\leq k\leq n-1}\HH^{*}(\mathscr{C},\mathbb{Q}) 
    = \HH^{*}(X,\mathbb{Q}) \oplus\bigoplus_{1\leq k\leq n-1} \HH^{*}(C,\mathbb{Q}).$$
\end{lemma}

\begin{definition}\label{chorb}
Let $\operatorname{pr}_{i}, 1\leq i\leq n-1,$ be the $i$-th projection from 
$$
K_{0}(\uX)=K_{0}(C)\rho_{1}\oplus...\oplus K_{0}(C)\rho_{n-1}\oplus K_{0}(X)
$$
to $K_{0}(C)\rho_{i}$, where $K_{0}(-)$ is the Grothendieck group. We define
\begin{align*}
    \operatorname{ch}_{orb}: K_{0}(\uX)& \to \HH^{*}_{CR}(\uX,\mathbb{Q}) =\HH^{*}(X,\mathbb{Q})\bigoplus_{1\leq k\leq n-1} \HH^{*}(C,\mathbb{Q}) \\
    E &\mapsto (\operatorname{ch}(E), \operatorname{ch}(\operatorname{pr}_{1}(E)),...,\operatorname{ch}(\operatorname{pr}_{n-1}(E))). 
\end{align*}

\end{definition}

\subsection{Stable sheaves}
We collect some facts about stable sheaves on root surfaces. For general theory for Deligne-Mumford stacks, we recommend the readers to \cite{Nir08} and \cite{JK24} for details.

Let $X$ be a smooth projective surface, $C\subset X$ be a smooth curve, and $\uX$ be the root stack in Definition \ref{rootstack}. Note that $\uX$ is a smooth projective Deligne-Mumford stack of dimension 2, and $X$ is its coarse space.

\begin{definition}[\cite{Lie11}]
    A line bundle $L$ on $\uX$ is ample if some power of $L$ is the pullback of an ample line bundle on $X$ under the structure map $\pi: \uX \to X $.
\end{definition}

We use Vistoli's intersection theory on $\uX$ in what follows  \cite{Vis89}. In particular, Chern classes and Todd classes are defined, as well as a degree map. The Hodge index theorem still holds.

\begin{theorem}\textnormal{(Hodge index theorem)}
    Suppose $H$ is an ample divisor on $\uX$. If $D$ is a divisor such that $D\cdot H=0$ then $D^{2}\leq 0$. In particular, the intersection form on $\operatorname{NS}(\uX)\otimes{\mathbb{R}}$ is of signature $(1,  \rho-1)$.
\end{theorem}

\begin{proof}
See \cite[Theorem 3.1.3]{Lie11}
\end{proof}

\begin{definition}\label{slope stability}
For every $E\in \Coh(\uX)$ and every ample divisor $H$, the slope of $E$ with respect to $H$ is 
\[
\mu_{H}(E):=\left\{ 
\begin{array}{ll}
(H\cdot \operatorname{ch}_{1}(E))/\operatorname{rk}(E), & \text{if } \operatorname{rk}(E)>0, \\
+\infty, & \text{if } \operatorname{rk}(E)=0. 
\end{array}
\right.
\]

We say that $E$ is $\mu_{H}$-(semi)stable if for any non-zero proper subsheaf $F\subseteq E$ one has $\mu_{H}(F)\leqp \mu_{H}(E/F)$.
\end{definition}

The following Bogomolov inequality on $\uX$ is an important ingredient for the construction of Bridgeland stability conditions.

\begin{theorem} [\cite{JK24}]\label{JK}
 For every torsion free $\mu_{H}$-semistable sheaf $E$ on $\uX$, we have
 $$\Delta(E):=\operatorname{ch}_{1}(E)^{2}-2\operatorname{rk}(E)\operatorname{ch}_{2}(E)\geq 0. $$  
\end{theorem}
\begin{proof}
    See  proposition 3.6 and remark 3.7 of \cite{JK24}.
\end{proof}

\subsection{Bridgeland stability conditions}
In this subsection, we recall the notion of  Bridgeland stability condition. Some detailed references are \cite{Bri07,Bri08, MS17}.

\begin{definition}
    Let $\mathcal{D}$ be a triangulated category. We say that a full additive subcategory $\mathcal{A}\subset \mathcal{D}$ is the heart of a bounded $t$-structure if both of the following conditions are satisfied.
\begin{enumerate}
    \item If $i<0$ and $A,B\in \mathcal{A}$ then $\operatorname{Hom}_{\mathcal{D}}(A, B[i])=0$;
    \item For every nonzero object $E\in \mathcal{D}$ there exists integers $k_{1}>k_{2}>...>k_{m}$ and objects $E_{1},...,E_{m}\in \mathcal{D}$ that fit into the following collection of distinguished triangles
$$
\begin{tikzcd}
0=E_{0} \arrow[r]       & E_{1} \arrow[r] \arrow[ld] & E_{2} \arrow[r] \arrow[ld] & ... \arrow[r] & E_{m-1} \arrow[r]         & E_{m}=E. \arrow[ld] \\
A_{1}[k_{1}] \arrow[u, dashed] & A_{2}[k_{2}] \arrow[u, dashed]    &                        &               & A_{m}[k_{m}] \arrow[u, dashed] &          
\end{tikzcd}
$$
\end{enumerate}
\end{definition}

Note that the heart of a bounded $t$-structure is an abelian category.
Now we fix a finite rank lattice $\Lambda$ and a surjective group homomorphism $\vv: K_{0}(\mathcal{D})\twoheadrightarrow \Lambda$.

\begin{definition}\label{pre-stab}
    A \emph{pre-stability condition} (with respect to $\Lambda$) on $\mathcal{D}$ is a pair $\sigma=(Z, \mathcal{A}) $ where 
\begin{itemize}
    \item $\mathcal{A}$ is the heart of a bounded $t$-structure of $\mathcal{D}$;
    \item $Z: \Lambda\to \mathbb{C}$ is a group homomorphism, called the \emph{central charge} of $\sigma$.
\end{itemize}
They satisfy the following properties:
\begin{enumerate}
    \item For any non-zero $E\in \mathcal{A}$,

$$Z(E):= Z(\vv({E}))\in \mathbb{R}_{>0}\cdot e^{i\pi\phi}$$
with $\phi\in (0,1]$. Define the phase of $0\not =E\in \mathcal{A}$ to be $\phi(E):=\phi$. We say that $E\in \mathcal{A}$ is $\sigma$-\emph{(semi)stable} if for any non-zero proper subobject $F\in \mathcal{A}$ of $E$, $\phi(F)\leqp \phi(E)$.  We denote $\mathcal{P}(\phi)$ as the full subcategory of $\mathcal{A}$ whose objects are $\sigma$-semistable of phase $\phi$.
\item (\HN filtration) Every object $E\in \mathcal{A}$ admits a unique filtration
$$0=E_{0}\subseteq E_{1}...\subseteq E_{n-1}\subseteq E_{n}=E,$$
such that the quotients $E_{i}/E_{i-1}\in \mathcal{P}(\phi_{i})$ with $\phi_{1}>\phi_{2}>...>\phi_{m}$. Such filtration is called the \emph{Harder-Narasimhan filtration} of $E$ with respect to $\sigma$
\end{enumerate}
For every $E\in \cA$, we write $\phi_{\min}(E):=\phi_{m}$ and $\phi_{\max}(E):=\phi_{1}$. If a function $Z: \Lambda\to \mathbb{C}$ satisfies condition (i), we say $Z$ is a \emph{stability function}.
\end{definition}

\begin{definition}\label{quadratic}
    A pre-stability condition $\sigma=(Z, \cA)$ is a\emph{ stability condition} (with respect to $\Lambda$) if it additionally satisfies either of the following equivalent \emph{support property} (see e.g. \cite{MS17}):
\begin{itemize}
\item For any fixed norm $\|\cdot \|$ on $\Lambda_{\mathbb{R}}$, there is a constant $M>0$, such that for every $\sigma$-semistable object $E\in \cA$, we have 
$$ \| E \| \leq M |Z(E)|. $$
\item There exists a quadratic form $Q$ on $\Lambda_{\mathbb{R}}$ such that 
\begin{enumerate}
  \item $Q$ is negative definite on $\ker(Z)$, and
  \item $Q(E)\geq 0$ for every $\sigma$-semistable object $E\in \mathcal{A}$.
\end{enumerate}
\end{itemize}
\end{definition}

Let $\operatorname{Stab}_{\Lambda}(\mathcal{D})$ denote the set of stability conditions on $\mathcal{D}$ with respect to $\Lambda$.  This set can be given a topology as the coarsest topology such that for any $E\in \mathcal{D}$ the maps $(Z, \cA)\to Z, (Z, \cA)\to \phi_{\max}(E)$ and $(Z, \cA)\to \phi_{\min}(E)$ are continuous (see \cite{Bri07}). We have the following deformation result.

\begin{theorem}[\cite{Bri07}] \label{deformation}
Using the topology defined in \cite{Bri07}, the central charge map 
$$\mathcal{Z}: \operatorname{Stab}_{\Lambda}(\mathcal{D})\to \operatorname{Hom}(\Lambda, \mathbb{C}), \quad (Z, \mathcal{A})\mapsto Z$$
is a local homeomorphism. In particular, $\operatorname{Stab}_{\Lambda}(\mathcal{D})$ has a complex manifold structure of dimension $\operatorname{rk}(\Lambda)$.
    
\end{theorem}

\section{Calculation on root stack}\label{section3}
In this section, we recall the structure of $D^{b}(\uX)$ and compute necessary formulas for Chern classes of coherent sheaves on $\uX$. 

\subsection{Derived category of root stacks}
In this subsection we collect properties of $D^{b}(\uX)$.

Recall that we have the following embedding in (\ref{square}):
$$j: \mathscr{C}\to \uX.$$
By construction, we have $j^{*}(\underline{\mathcal{M}})\cong \mathcal{M},$  where $\underline{\mathcal{M}}$ is the universal line bundle on $\uX$ and $\mathcal{M}$ is the universal line bundle on $\mathscr{C}$ (Definition \ref{universal}). Note that $\underline{\mathcal{M}}=\mathcal{O}_{\uX}(\mathscr{C})$.
In the following, we will not distinguish $\underline{\rho_{k}}=\underline{\mathcal{M}}^{-k}$ on $\uX$ and $\rho_{k}=\mathcal{M}^{-k}$ on $\mathscr{C}$. For simplicity, we will use $\rho_{k}$ for both $\rho_{k}$ and $\underline{\rho_{k}}$.

We need the following observation.
 \begin{lemma}\label{cotangent bundle}
Using notations in (\ref{square}), let $F\in \Coh(C)$. Then \begin{equation}
    \cH^{i}(\LL j^{*}j_{*}p^{*}F )      =
    \begin{cases}
     p^{*}F & \text{$i=0$,}\\
     p^{*}F\rho_{1}  & \text{$i=-1$,}\\
     0 & \text{otherwise.}
    \end{cases}       
\end{equation}
\end{lemma}

\begin{proof}
    See  \cite[Proposition 6.1]{IU15}.
\end{proof}

The following lemma is a refinement of Theorem \ref{semiorthgonal}.

 \begin{lemma}\label{circle}
			We have the following:
			\begin{itemize}
			\item $\Hom_{\uX}(j_{*}D^{b}(C)\rho_{k}, D^{b}(X))=0$ for all $k\neq n-1$. 
			\item $\Hom_{\uX}(j_{*}D^{b}(C)\rho_{k}, j_{*}D^{b}(C)\rho_{l})=0$ for all $l \neq k, k+1$. 
		\end{itemize}
\end{lemma}
		
		\begin{proof}
        To show the first claim, by Serre duality we have 
\begin{align*}
\Hom_{\uX} (j_{*}D^{b}(C)\rho_{k}, D^{b}(X)) &=\Hom_{\uX}(D^{b}(X), j_{*}D^{b}(C)\rho_{k-(n-1)})^{*}\\
     &= \Hom_{\mathscr{C}}(\LL j^{*} D^{b}(X), D^{b}(C)\rho_{k-(n-1)})^{*}.
\end{align*}        
        Note that $\LL j^{*} D^{b}(X)\subset D^{b}(C)\rho_{0}$, the claim is proved. 
        To see the second claim, note that 
        $$\Hom_{\uX}(j_{*}D(C)\rho_{k}, j_{*}D(C)\rho_{l})=\Hom_{\mathscr{C}}(\LL j^{*} j_{*}(D^{b}(C)\rho_{k}), D^{b}(C)\rho_{l}). $$
        By Lemma \ref{cotangent bundle}, we have
        $\LL j^{*} j_{*}(D^{b}(C)\rho_{k})\subset \langle D^{b}(C)\rho_{k}, D^{b}(C)\rho_{k+1} \rangle. $
        \end{proof}

\subsection{Coherent sheaves on root stacks}
In this subsection we study structure of coherent sheaves on root stacks. 

We first show the exactness of $\pi_{*}$ and $\pi^{*}$ in the following lemmas.

\begin{lemma}[{\cite[Lemma 2.3.4]{AV02}}]
Let $\pi: \uX \to X$ be the coarse space. The functor $\pi_{*}$ maps quasicoherent sheaves to quasicoherent sheaves, coherent sheaves to coherent sheaves, and is exact. Moreover, $\pi_{*}(\mathcal{O}_{\uX})=\mathcal{O}_{X}.$

\end{lemma}

\begin{lemma}
    $\pi_{}$ is flat, that is $\pi^{*}$ is exact.
\end{lemma}
\begin{proof}
    We check locally. By Proposition \ref{rootstackDM}, locally we have following commutative diagram
\begin{center}
\begin{tikzcd}
\operatorname{Spec}(A[x]/(x^{n}-a)) \arrow[r, "\varphi"] \arrow[rd, "\psi"] & \lbrack \operatorname{Spec}(A[x]/(x^{n}-a))/\mu_{n} \rbrack \arrow[d, "\pi"] \\
                                 & \operatorname{Spec}A.               
\end{tikzcd}
\end{center}
Since $\psi$ is finite flat and $\varphi$ is finite \'etale, $\pi$ is flat.
\end{proof}

For torsion free sheaves we have the following injective map.
\begin{lemma}\label{injective}
    If $E$ is torsion free on $\uX$, then the adjunction map $ \pi^{*}\pi_{*}E\to E$  is injective.
\end{lemma}

\begin{proof}
    First note that $\pi_{*}E$ is torsion free. Otherwise there is a torsion subsheaf $F\hookrightarrow \pi_{*}E$. Hence $$\operatorname{Hom}_{\uX}(\pi^{*}F, E)=\operatorname{Hom}_{X}(F, \pi_{*}E)\not =0.$$
However $\pi^{*}F$ is torsion, $\operatorname{Hom}_{\uX}(\pi^{*}F, E)=0$, we get a contradiction.

Now since $\psi^{*}\pi_{*}E=\varphi^{*}(\pi^{*}\pi_{*}E)$ is torsion free, $\pi^{*}\pi_{*}E$ is torsion free.
Since $\pi$ is isomorphism away from $C$, $\operatorname{Supp}(\operatorname{Ker}(\varphi))\subseteq \mathscr{C}$. So $\operatorname{Ker}(\varphi)=0$ because $\pi^{*}\pi_{*}E$ is torsion free. Thus the adjunction map $ \pi^{*}\pi_{*}E\to E$  is injective.
\end{proof}

Next we study the relation between $\Coh(\uX)$ and $\Coh(\sX)$.
Let $\mathscr{X}$ be the root gerbe in Definition \ref{rootstack}. Now we construct a morphism $f: \uX \to \mathscr{X} $. Note $X$ has an open affine cover $$X=\bigcup_{i=1}^{n}\operatorname{Spec}(A_{i}),$$ such that such that $\mathcal{O}_{X}(C)|_{U_{i}}\cong \mathcal{O}_{U_{i}}$. Choosing $a_{i}\in A_{i}$ such that $a_{i}$ defines the curve $C|_{U_{i}}$.
    By construction,  $$\uX=\bigcup_{i=1}^{n}[\operatorname{Spec}(A_{i}[x]/(x^{n}-a_{i}))/\mu_{n}],\ \mathscr{X}=\bigcup_{i=1}^{n}[\operatorname{Spec}(A_{i})/\mu_{n}]. $$

The natural map $\operatorname{Spec}(A_{i}[x]/(x^{n}-a_{i})) \to \operatorname{Spec}(A_{i})$ is $\mu_{n}$-equivariant, thus it induces a map
$$f_{i}:\bigcup_{i=1}^{n}[\operatorname{Spec}(A_{i}[x]/(x^{n}-a_{i}))/\mu_{n}]\to\bigcup_{i=1}^{n}[\operatorname{Spec}(A_{i})/\mu_{n}]. $$
Note that $\{f_{i}\}$ can be glued since they agree on the overlap, we make the following definition.
\begin{definition}\label{mapf}
   We define $f: \uX \to \mathscr{X} $ to be the morphism glued by $\{f_{i}\}$ as above.
\end{definition}
Note that we have the following commutative diagram 
\begin{center}
 \begin{tikzcd}
\uX \arrow[rr, "f"] \arrow[rd, "\pi"] &   & \mathscr{X} \arrow[ld, "p_{X}"] \\
                                  & X. &                  
\end{tikzcd}
\end{center}
By construction, $f$ is a finite flat map of degree $n$ ramified along $\mathscr{C}$. Moreover, $f_{*}(\mathcal{O}_{\uX})=\bigoplus_{k=0}^{n-1}\mathcal{O}_{\mathscr{X}}\rho_{k}.$

By \cite[Theorem 2.22]{Nir08a}, 
Serre duality holds for $\uX$ and the dualizing object is $\omega_{\uX}[2]$. By \cite[Proposition 3.4]{Nir08a}, the canonical bundle $\omega_{\uX}$ on $\uX$ is 
$$\omega_{\uX}=\pi^{*}\omega_{X}\otimes \OO_{\uX}((n-1)\mathscr{C}). $$

\begin{proposition}\label{SD}
	We have the following Serre dualities:
	\begin{itemize}
		\item $\Hom_{D^{b}(\mathscr{X})}(A, B)=\Hom_{D^{b}(\mathscr{X})}(B, A\otimes \omega_{X}[2])^{*}$. 
		\item $\Hom_{D^{b}(\uX)}(A, B)=\Hom_{D^{b}(\uX)}(B, A\otimes \omega_{X}((n-1)\mathscr{C})[2])^{*}$. 
	\end{itemize}
	In particular, we have 
	$$\pi^{!}(-)=\pi^{*}(-)\otimes \OO_{\uX}((n-1)\mathscr{C}), \quad f^{!}(-)=f^{*}(-)\otimes \OO_{\uX}((n-1)\mathscr{C}). $$
\end{proposition}

\begin{proof}
	The first claim follows from the fact that $D^{b}(\mathscr{X})=\bigoplus_{k=0}^{n-1} D^{b}(X)\rho_{k}$. Since Serre duality is true for $\uX$ and the dualizing object is $\omega_{\uX}[2]$ by \cite{Nir08a}, the second claim is true.  Then $\pi^{!}$ and $f^{!}$ are  computed by the relative dualizing objects.
\end{proof}

    \begin{lemma}\label{chain}
            Let $G$ be a sheaf in $D^{b}(X)^{\perp} \cap \operatorname{Coh}(\uX)$, that is $G$  has a filtration as $G_{k}\rho_{k}$ where $G_{k}\in \operatorname{Coh}(C)$.
            Then 
            $$f_{*}G=\bigoplus_{k=1}^{n-1} G_{k}\rho_{k}.$$
        Here, we identify the sheaf $G$ on $C$ with $\bar{j}_{*}G$ on $X$. 
            
        \end{lemma}
    
        \begin{proof}

The following two diagrams are commutative:
\begin{center}
\begin{tikzcd}
\sC \arrow[r, "j"] \arrow[d, "p"] & \uX \arrow[r, "f"] \arrow[d, "\pi"] & \mathscr{X} \arrow[ld, "p_{X}"] & \sC \arrow[r, "f\circ j"] \arrow[d, "p"] & \mathscr{X}, \arrow[d, "p_{X}"] \\
C \arrow[r, "\bar{j}"]                & X                               &                   & C \arrow[r, "\bar{j}"]                & X,              
\end{tikzcd}
\end{center}
   where the right-hand side diagram is a Cartesian diagram.         
Note the $G_{k}\rho_{k}$ on $\uX$ denotes $j_{*}p^{*}G_{k}\rho_{k}$, and $\rho_{k}$ on $\sC$ is the pullback of $\rho_{k}$ on $\mathscr{X}$. Since $p_{X}$ is flat,  by base change and projection formula, we have 
$$f_{*}G_{k}\rho_{k}= (f\circ j)_{*}(p^{*}G_{k}\otimes (f\circ j)^{*}\rho_{k})=p_{X}^{*}\bar{j}_{*}G_{k}\rho_{k}. $$
Since $\operatorname{Coh}(\mathscr{X})=\bigoplus \operatorname{Coh}(X)\rho_{k}$, we have 
$$f_{*}G=\bigoplus_{k=1}^{n-1}p_{X}^{*}\bar{j}_{*}G_{k}\rho_{k}=\bigoplus_{k=1}^{n-1}G_{k}\rho_{k}.$$
        \end{proof}\
        
Now we recall a well-known fact about torsion free sheaves on root stacks.
        
\begin{theorem}[\cite{BV12}]\label{parabolic}
    The category $\mathrm{TF}(\uX)$ of torsion-free sheaves on $\uX$ is equivalent to the category $\mathrm{Par}^{n}(X)$ of $n$-parabolic sheaves on $X$. 
\end{theorem}
For any torsion free sheaf $E\in \Coh(\uX)$, we compute the projections with respect to the decomposition
$ D^{b}(\uX)=\langle D^{b}(C)\rho_{1}, \cdots, D^{b}(C)\rho_{n-1}, D^{b}(X) \rangle$, in terms of the parabolic structure in Theorem \ref{parabolic}.
        Consider the exact sequence
        $$ 0 \longrightarrow f_{*}(E(-\mathscr{C}) )\longrightarrow f_{*}(E) \longrightarrow f_{*}(E|_{\mathscr{C}}) \longrightarrow 0. $$
        For the rest of this paper, we write 
        \begin{equation}\label{push}
         f_{*}(E)=\bigoplus_{k=0}^{n-1} E_{k}\rho_{k}. 
        \end{equation}
        Note that $\OO_{\uX}(-\mathscr{C})=f^{*}(\OO_{\mathscr{X}}\rho_{1})$. By the projection formula, we have
        $$f_{*}(E(-\mathscr{C}) )= \bigoplus _{k=0}^{n-1}E_{k}\rho_{k+1}=E_{n-1}(-C)\rho_{0}\oplus \left( \bigoplus_{k=1}^{n-1}E_{k-1}\rho_{k} \right) . $$
        Hence we have morphisms
        $$
        e_{-1, 0}: E_{n-1}(-C) \to E_{0}, \quad e_{k-1, k}: E_{k-1} \to E_{k}, ~1 \leq k \leq n-1.
        $$
By a local computation, we see the maps $e_{k-1, k}$ are precisely those in the definition of parabolic sheaves on $\uX$ (see Theorem \ref{parabolic}). For every $n\in \ZZ$, let $e_{k+n-1, k+n}=e_{k-1, k}\otimes \OO_{X}(nC)$. For integers $j<k$, let $e_{jk}$ be the composition of morphisms
\begin{equation}\label{ejk}
e_{jk}:=e_{k-1, k}\circ \cdots e_{j+1, j+2}\circ e_{j, j+1}, \quad G_{jk}:= \mathrm{coker}(e_{jk}).
\end{equation}
For $0\leq k \leq n-1$, we define $E_{0,k}$ to be the torsion free sheaf on $\uX$ that corresponds to the parabolic sheaf
\begin{equation}\label{filters}
E_{0}\overset{\id}{\too} \cdots \overset{\id}{\too} E_{0} \overset{e_{0,k}}{\too} E_{k} \overset{e_{k, k+1}}{\too} \cdots \overset{e_{n-2, n-1}}{\too} E_{n-1} \overset{e_{n-1, n}}{\too} E_{0}(C),
\end{equation}
where $E_{k}$ is at the $k$-th position.
 Summarizing these, we have the following proposition.
\begin{proposition}\label{link}
    Let $E\in \Coh(\uX)$ be a torsion free sheaf. Then there is a filtration
    $$E_{0} \overset{f_{n, n-1}}{\hookrightarrow} E_{0,n-1} \overset{f_{n-1,n-2}}{\hookrightarrow} E_{0,n-2} \overset{f_{n-2, n-3}}{\hookrightarrow} \cdots \overset{f_{2, 1}}{\hookrightarrow} E_{0,1} = E.$$
    Using notations in (\ref{ejk}) and (\ref{filters}), we have the graded factor $\mathrm{coker}(f_{k+1, k})=G_{0,k}\rho_{k}$ for $1\leq k \leq n-1$. 
    
    For every $1\leq k \leq n-1$, the sheaf $G_{0, k}\rho_{k}$ is the projection of $E$ into $D^{b}(C)\rho_{k}$ under the semiorthogonal decomposition (Lemma \ref{semiorthgonal}). Under $f_{*}$, the map $f_{k+1, k}$ is the identity on the $\rho_{j}$ component for $j\neq k$ and is $e_{k, k+1}$ on the $\rho_{k}$ component.
\end{proposition}

To end this subsection we make the following notation simplification.
		\begin{lemma}\label{commute}
			Let $E\in D^{b}(\uX)$ and $f: \uX \to \sX$ be the map in Definition \ref{mapf}. Write $f_{*}(E)=\bigoplus_{k=0}^{n-1}E_{k}\rho_{k}$. Then $\cH^{i}(E_{k})=\cH^{i}(E)_{k}$. 
		\end{lemma}
		
		\begin{proof}
			The claim follows from the exactness of $\pi_{*}$. 
		\end{proof}
    In the following of this paper we will not distinguish $\cH^{i}(E_{k})$ and $\cH^{i}(E)_{k}$.

 \subsection{Chern class computations}
In this subsection we make computations for Chern classes.

First, the Chow ring of a root stack is computed by  \cite{arena2023integral} as follows.
\begin{lemma}
	The Chow ring of $\uX$ is
		\begin{equation}
				\CH^{*}(\uX)= \dfrac{\CH^{*}(X)\oplus \CH^{*}(\PP(N_{C/X}, d))e}{(i_{*}\alpha-d\alpha e)_{\forall \alpha \in \CH^{*}(C)}}= \dfrac{\CH^{*}(X)\oplus \CH^{*}(C)[t]e}{((nt-D)e, i_{*}\alpha- n\alpha e)_{\forall \alpha \in \CH^{*}(C)}}. 
			\end{equation}
			The numerical ring of $\uX$ is 
			\begin{equation}
				\Num^{*}(\uX)=\Num^{*}(X)\left[\sC, \mathbf{B}\mu_{n}\right].
			\end{equation}
			Here, $\sC=[\sC]$ is the fundamental class of the root gerbe $\sC$. Its numerical class is $\dfrac{C}{n}$. $\mathbf{B}\mu_{n}=[\mathbf{B}\mu_{n}]$ is the fundamental class of the gerbe over a point. Its numerical class is $\dfrac{\mathrm{pt}}{n}$. 
		\end{lemma}
The next lemma computes the Chern classes of $E\in \Coh(\sC)$.
		\begin{lemma}\label{uch2}
	Let $E_{k}\in \Coh(C)\rho_{k}\subset \Coh(\sC)$ with $\ch_{0}(C, E_{k})=r_{k}$ and $\ch_{1}(C, E)=d_{k}$. Then we have
		\begin{equation}
			\ch(\uX, j_{*}(E_{k})\rho_{k} )= (0, r_{k}[\sC] , r_{k}\dfrac{-2k-1}{2}\sC^{2} + d_{k}[\mathrm{B}\mu_{n}]). 
		\end{equation}
	\end{lemma}	
\begin{proof}
    The claim follows from the Riemann-Roch on Deligne-Mumford stacks \cite{Toe99}. 
\end{proof}
 The following lemma is used multiple times in later sections.
	
		\begin{lemma}\label{chlink}
			Suppose $f_{*}(E)=\bigoplus_{k=0}^{n-1} E_{k}\rho_{k}$. Numerically, we have
			\begin{align*}
				\ch_{0}(E)& =\dfrac{1}{n}\sum_{k=0}^{n-1} \ch_{0}(E_{k})=\ch_{0}(E_{j}), \forall j, \\
				\ch_{1}(E)& = \dfrac{1}{n}\sum_{k=0}^{n-1} \ch_{1}(E_{k}), \\
				\ch_{2}(E)& =\dfrac{1}{n}\sum_{k=0}^{n-1} \ch_{2}(E_{k}) + \sum_{k=0}^{n-1} \dfrac{n-2k-1}{2n^{2}}\ch_{1}(E_{k})\cdot C.        		
			\end{align*}
		\end{lemma}
		
		\begin{proof}
			We have $\ch_{0}(E_{i})=\ch_{0}(E_{j}), \forall i,j$, since $f_{*}$ is \'{e}tale away from $\sC$. Hence the first claim is proved. By Proposition \ref{link},  and Lemma \ref{uch2}, we have
			\begin{align*}
				\ch_{1}(E)&=\ch_{1}(E_{0})+\dfrac{1}{n}\sum_{k=1}^{n-1}(\ch_{1}(E_{k})-\ch_{1}(E_{0}) )
				=  \dfrac{1}{n}\sum_{k=0}^{n-1} \ch_{1}(E_{k}), \\
				\ch_{2}(E)&=\ch_{2}(E_{0})+\sum_{k=1}^{n-1}\left[\dfrac{\ch_{2}(G_{0,k})}{n}+ \dfrac{\ch_{1}(G_{0,k})\cdot C}{2n} - \dfrac{2k+1}{2n^{2}}\ch_{1}(G_{0,k})\cdot C \right]\\
				&= \dfrac{1}{n}\sum_{k=0}^{n-1} \ch_{2}(E_{k}) + \sum_{k=0}^{n-1} \dfrac{n-2k-1}{2n^{2}}\ch_{1}(E_{k})\cdot C.        
			\end{align*}
		\end{proof}
 Since Chern classes computations in this paper are all numerical, we will not distinguish $[\sC], \dfrac{C}{n}$ and $ [\mathbf{B}\mu_{n}], \dfrac{\mathrm{pt}}{n}$. 
  As consequences of Proposition \ref{link}, we have the following inequalities.

	\begin{lemma}\label{ranks}
			Let $E\in \Coh(\uX)$ be a torsion free  sheaf. Let $G_{j,k}$ be the sheaf defined in (\ref{ejk}) and $r_{j,k}:=\ch_{0}(C, G_{j,k})$, where $0\leq j, k\leq n-1$. Then 
   \begin{enumerate}

			 \item  $0\leq r_{0,1}\leq r_{0,2}\leq \cdots \leq r_{0, n-1}\leq \ch_{0}(E). $

              \item $r_{j,k}\leq \ch_{0}(E)$ for any $ 0\leq  j,k \leq n-1 . $
   \end{enumerate}
		\end{lemma}
		
		\begin{proof}
			Note $G_{0,k}$ has a composition series $G_{0,1}, G_{1,2}, \cdots, G_{k-1,k}$. By Lemma \ref{chlink}, we have 

                      \begin{align*}
                     \ch_{1}(E|_{\sC}) &= \ch_{1}(E)-\ch_{1}(E(-\sC))\\
                       &=\frac{1}{n}\sum_{k=0}^{n-1}\ch_{1}(E_{k})-\frac{1}{n}(\ch_{1}(E_{n-1}(-C))+ \sum_{k=0}^{n-2}\ch_{1}(E_{k})) \\
                       &=\frac{1}{n}(\ch_{1}(E_{n-1})-\ch_{1}(E_{n-1}(-C)))\\
                       &=\frac{1}{n}(\sum_{k=1}^{n-1}(\ch_{1}(E_{k})-\ch_{1}(E_{k-1}))+\ch_{1}(E_{0})-\ch_{1}(E_{n-1}(-C)))
                       =\frac{1}{n}(\sum_{k=0}^{n-1}\ch_{1}(G_{k-1,k})).
             \end{align*}  
             Since $G_{k-1,k}$ are supported on the curve $C$, we have following:
			$$ 
			\ch_{0}(E)\sC=\ch_{1}(E|_{\sC})=\frac{1}{n}\sum_{k=0}^{n-1} \ch_{1}(G_{k-1,k})=\frac{1}{n}\sum_{k=0}^{n-1} \ch_{0}(C, G_{k-1,k})C. $$
   Hence $r_{0,k}=\ch_{0}(C, G_{0,k}) =\sum_{j=0}^{k}\ch_{0}(C, G_{j,j+1}) \leq \ch_{0}(E)$, $r_{0,k}\leq r_{0, k+1}$, and $r_{j,k}\leq \ch_{0}(E)$.

   \end{proof}

\begin{comment}
		Note that for every $0\leq k \leq n-1$ we have an injection $E_{k}(-k\sC) \to E$ on $\uX$ induced by the map $E_{k}\rho_{k} \to f_{*}(E)$ on $\sX$. The following lemma will be used later.
		
		\begin{lemma}
			For every $0\leq k \leq n-1$, we have $\mu_{H}(E_{k}(-k\sC)) \leq \mu_{H}(E)$ for every torsion free sheaf $E$.
		\end{lemma}
		
		\begin{proof}
			Since $\ch_{0}(E_{k}(-k\sC))=\ch_{0}(E)$, it suffices to compare their $H$-degrees. Write $$f_{*}(E(k\sC))=\bigoplus_{j=0}^{n-1} E'_{j}\rho_{j}, $$
			then we have $E'_{0}=E_{k}$. By Proposition \ref{link} and Corollary \ref{ranks}, we have $\ch_{1}(E'_{j})\cdot H \geq \ch_{1}(E'_{0}) \cdot H$. By Lemma \ref{chlink}, we have
			$$ \ch_{1}(E(k\sC))\cdot H=\dfrac{1}{n}\sum_{j=0}^{n-1}\ch_{1}(E'_{j})\cdot H\geq \ch_{1}(E'_{0})\cdot H=\ch_{1}(E_{k})\cdot H. $$
			Hence $\ch_{1}(E)\cdot H \geq \ch_{1}(E_{k}(-k\sC)\cdot H$ and then $ \mu_{H}(E_{k}(-k\sC)) \leq \mu_{H}(E)$
   \end{proof}

\end{comment}

\section{Construction of stability conditions}\label{section4}

 \begin{lemma}
     The abelian category $\operatorname{Coh}(\uX)$ is Noetherian.
 \end{lemma}     

 \begin{proof}
    By construction we have an affine cover
    $$\uX=\bigcup_{i=1}^{n}V_{i},$$
where $V_{i}=[\operatorname{Spec}(A_{i}[x]/(x^{n}-a_{i}))/\mu_{n}]$. Since we have the following equivalence of categories 
$$\Coh(V_{i}) \simeq \Coh^{\mu_{n}}(\operatorname{Spec}(A_{i}[x]/(x^{n}-a_{i}))),$$
the category
$\operatorname{Coh}(V_{i})$ is Noetherian. Thus $\operatorname{Coh}(\uX)$ is Noetherian. 

\end{proof}

\begin{lemma}
    For any coherent sheaf $E$ on $\uX$, it has Harder-Narasimhan filtration with respect to the slope stability (Definition \ref{slope stability}).
\end{lemma}

\begin{proof}
Consider the function $$Z: K_{0}(\operatorname{Coh}(\uX))\to \mathbb{C}, \quad Z(E)=- c_{1}(E)\cdot H+\operatorname{rk}(E)i  .$$

Since $\operatorname{Coh}(\uX)$ is Noetherian and $\Im(Z(K_{0}(\uX)))$ is discrete, by \cite[Proposition 4.10]{MS17}, every $E\in \operatorname{Coh}(\uX)$ has the Harder-Narasimhan filtration. 
\end{proof}
 Recall that for $B\in \NS(\uX)_{\RR} = \NS(X)_{\RR}$, the \emph{twisted Chern character} is defined as  $$\ch^{B}: K_{0}(\uX) \rightarrow \HH^{*}(\uX, \RR), \quad \ch^{B}(E):= e^{-B}\cdot \ch(E). $$
        Explicitly, this is
\begin{align*}
    \ch_{0}^{B}(-) &=\ch_{0}(-),\\
\ch_{1}^{B}(-) &= \ch_{1}(-)-B\cdot \ch_{0}(-), \\
\ch_{2}^{B}(-) &= \ch_{2}(-)-B\cdot \ch_{1}(-)+\dfrac{B^{2}}{2}\cdot \ch_{0}(-). 
\end{align*}

\begin{lemma}\label{twisted chern}
    Lemma \ref{chlink} holds for twisted Chern character.
\end{lemma}

\begin{proof}

It is clear for $\ch_{0}^{B}$ and $\ch_{1}^{B}$ parts. For $\ch_{2}^{B}$ part, since $\sum_{k=0}^{n-1}\frac{n-2k-1}{2n^{2}}=0$,  we have following:
\begin{align*}
\ch_{2}^{B}(E) &= \ch_{2}(E)-B\cdot \ch_{1}(E)+\dfrac{B^{2}}{2}\cdot \ch_{0}(E)\\
                          &=\dfrac{1}{n}\sum_{k=0}^{n-1} \ch_{2}(E_{k}) + \sum_{k=0}^{n-1} \dfrac{n-2k-1}{2n^{2}}\ch_{1}(E_{k})\cdot C-\dfrac{1}{n}B\sum_{k=0}^{n-1} \ch_{1}(E_{k})+\frac{B^{2}}{2}\frac{1}{n}\sum_{k=0}^{n-1}\ch_{0}(E_{k})\\
                          &=\frac{1}{n}\sum_{k=0}^{n-1}(\ch_{2}(E_{k})-B\ch_{1}(E_{k})+\frac{B^{2}}{2}\ch_{0}(E_{k}))+ \sum_{k=0}^{n-1} \dfrac{n-2k-1}{2n^{2}}(\ch_{1}(E_{k})-B\ch_{0}(E_{k}))\cdot C\\
                          &= \dfrac{1}{n}\sum_{k=0}^{n-1} \ch_{2}^{B}(E_{k}) + \sum_{k=0}^{n-1} \dfrac{n-2k-1}{2n^{2}}\ch_{1}^{B}(E_{k})\cdot C.                      \end{align*}

\end{proof}

        \begin{definition}\label{heart}
        Let $B, H \in \NS(\uX)_{\RR}$ with $H \in \Amp(\uX)$ an ample class. 
        	We define the following full subcategories of $D^{b}(\uX)$:
        		\begin{align*}\label{tilt}
        			\mathcal{T}_{B, H}:= &\{E \in \Coh(\uX): \mu_{H, \min}(F)>B\cdot H \}, \\
        			\mathcal{F}_{B, H}:= &\{E \in \Coh(\uX): \mu_{H, \max}(F)\leq B\cdot H \}, \\
        			\mathcal{A}_{B, H}:= & \langle \mathcal{F}_{B, H}[1], \mathcal{T}_{B, H} \rangle. 
        		\end{align*}
        \end{definition}

\begin{lemma}
        The category $\mathcal{A}_{B, H}$ is the heart of a bounded t-structure on $D^{b}(\uX)$. 
        \end{lemma}

\begin{proof}
    The proof follows from \cite[Lemma 6.1]{Bri07}.
\end{proof}

\begin{definition}\label{lattice}
Let $\Lambda$ be the image of $K_{0}(\uX)$ in $\HH^{*}_{CR}(\uX,\mathbb{Q})$ under $\operatorname{ch}_{orb}(-)$ (Definition \ref{chorb}), and $\Lambda'$ be the image of $K_{0}(\uX)$ in $\HH^{*}_{CR}(\uX,\mathbb{Q})$ under $\ch(-)$.
\end{definition}

Define the following symmetric paring on $\HH^{*}(\uX)$: 
        \begin{equation*}
        	\langle \uu, \vv \rangle := \int_{\uX} \uu \cdot \vv.
        \end{equation*}
        The following map induces a group homomorphism on $\HH^{*}(\uX)$:   
        \begin{align*}
Z_{B, H}(E):= -\langle e^{-(B+iH)},\ch(E) \rangle
        	&=\langle (-1+\dfrac{H^{2}}{2})+iH, \ch(E)\cdot e^{-B} \rangle \\
     &=-\ch_{2}^{B}(E) + \dfrac{H^{2}}{2}\ch_{0}^{B}(E)
        	+ i H\cdot (\ch_{1}^{B}(E)).
\end{align*}       

\begin{definition}\label{tilt stability}

We call $\sigma_{B,H}:=(Z_{B, H}, \cA_{B, H} )$ a \emph{tilt stability condition}.
    
\end{definition}

        \begin{proposition}\label{stability function}
        	 The group homomorphism $Z_{B, H}$ is a stability function (Definition \ref{pre-stab}) on the heart $\mathcal{A}_{B, H}$.
        \end{proposition}
        
        \begin{proof}
        	Let $E\in \mathcal{A}_{B, H}$ be a non-zero object. By definition, $\Im(Z_{B, H}(E))\geq 0$. If $\Im(Z_{B, H}(E))=0$, then $\mathcal{H}^{0}(E)$ has zero dimensional support, and $\mathcal{H}^{-1}(E)$ is torsion free and $\mu_{H}$-semistable with slope $B\cdot H$. Now assume $E$ is such an object. 
         
         Recall that for every class $\vv\in \HH^{*}(\uX)$, $\Delta(\vv):=\ch_{1}(\vv)^{2}-2\ch_{0}(\vv)\ch_{2}(\vv)$. Since $\Delta(\OO(-B))=0$, we have
$$ \ch_{1}^{B}(\vv)^{2}-2\ch_{0}^{B}(\vv)\ch_{2}^{B}(\vv)=\Delta(\vv\otimes \OO(-B))
=\Delta(\vv)+\Delta(\OO(-B))=\Delta(\vv). $$
         By the Bogomolov inequality (Theorem \ref{JK}), we have 
        	$\Delta(\mathcal{H}^{-1}(E)) \geq 0$.
        	If $\mathcal{H}^{-1}(E)\neq 0$, it has positive rank. By the Hodge Index Theorem, we have
        	$$\ch_{2}^{B}(\mathcal{H}^{-1}(E))\leq \dfrac{\ch_{1}^{B}(\mathcal{H}^{-1}(E))^{2}}{2\ch_{0}^{B}(\mathcal{H}^{-1}(E))} 
        	\leq 0. $$
        	Hence we have 
        \begin{align*}
        Z_{B, H}(E)=&-Z_{B, H}(\mathcal{H}^{-1}(E))+Z_{B, H}(\mathcal{H}^{0}(E))\\
        =&
        	\ch_{2}^{B}(\mathcal{H}^{-1}(E)) - \dfrac{H^{2}}{2}\ch_{0}^{B}(\mathcal{H}^{-1}(E)) - \ch_{2}^{B}(\mathcal{H}^{0}(E))
         \leq \dfrac{-H^{2}}{2} \ch_{0}(\mathcal{H}^{-1}(E))-\ch_{2}(\mathcal{H}^{0}(E)).
        \end{align*}
        	By assumption, $-H^{2}<0$. Since $E\neq 0$, either $\ch_{0}(\mathcal{H}^{-1}(E))>0$ or $\ch_{2}(\mathcal{H}^{0}(E))>0$. Hence $Z_{B, H}(E)<0$. 

        \end{proof}      

\begin{corollary}\label{prestability condition}
When $B, H\in \operatorname{NS}(X)_{\mathbb{Q}}$, $\sigma_{B, H}=(Z_{B, H}, \mathcal{A}_{B, H})$ is a pre-stability condition.
\end{corollary}

\begin{proof}
    By \cite{MS17},  we know $\mathcal{A}_{B, H}$ is Noetherian.
By assumption, $\Im(Z(\Lambda))$ is discrete. Thus by \cite[Proposition 4.10]{MS17}, $\sigma_{B, H}$ has the Harder-Narasimhan property. Since $\Im(Z(\Lambda))$ is discrete, by \cite[Lemma 4.4]{Bri08}, $\sigma_{B, H}$ is locally finite. So $\sigma_{B, H}$ is a pre-stability condition.
\end{proof}

To end this section we introduce several notations that is convenient for later computations.
Suppose $E\in \Coh(\uX)$, we define $\mu^{B}(E)$ as $$\mu^{B}(E):=\frac{\ch_{1}^{B}(E)\cdot H}{\ch_{0}^{B}(E)}.$$
In the following, fix $B, H\in \NS(\uX)_{\QQ}$ where $H$ is ample. Let $\cA=\cA_{B, H}$, and
		\begin{equation}\label{Zt}
			Z_{t}(\vv)=Z_{B,H , t}(\vv)=-\ch_{2}^{B}(\vv) + t\ch_{0}^{B}(\vv)
			+ i H \cdot (\ch_{1}^{B}(\vv)). 
		\end{equation}
		By Proposition \ref{stability function} and Corollary \ref{prestability condition}, for every $t>0$, $\sigma_{t}=(Z_{t}, \cA)$ is a pre-stability condition.
\begin{definition}
			Let $E\in D(\uX)$ and $\sigma$ be a pre-stablity condition on $\uX$. We define
			$$
			\mu_{\sigma}(E):=-\dfrac{\Re(Z(E))}{\Im(Z(E))}. 
			$$
		\end{definition}
Note that for $E, F\in \cA$, $\phi(E)>\phi(F)$ if and only if $\mu_{\sigma}(E)> \mu_{\sigma}(F)$.

\section{Support property}\label{section5}
	In this section, we  prove the support property of pre-stability conditions $\sigma_{t}$ for $t>0$. The main theorem of this paper is the following. 
		
		\begin{theorem}\label{maintheorem}
			For every $t>0$, the pre-stability condition $\sigma_{t}$ satisfies the support property with respect to $\HH^{even}_{CR}(\uX, \QQ)$.
		\end{theorem}
   We sketch the strategy to prove Theorem \ref{maintheorem}.
   It appears to the authors that writing down a quadratic form as in Definition \ref{quadratic} is very difficult.
   We use a different approach.
   
   First, by a numerical computation the problem is reduced to finding a bound for $\| \ch(E_{k}) \|, 0\leq k \leq n-1$ (Lemma \ref{normequivalence}), where the notation $E_{k}$ is defined in (\ref{push}). 
   
   We have the support property on the ordinary cohomology of $\uX$.
   For every $\sigma$-semistable object $E$, we will cut it into two parts
   $$0 \too P \too E \too Q \too 0.$$
   Here, $Q$ is a relatively simple part to deal with, however it prevents us to apply our observation directly to $E$. 
   Behaving differently, $P$ is our central problem, but we can delicately construct subobjects of $P$ so that their Chern classes are related to $P_{k}, 0\leq k \leq n-1$, where we write $f_{*}(P)=\bigoplus P_{k}\rho_{k}$ as in Proposition \ref{link}. By stability of $E$, we will get bounds for $\| \ch(P_{k}) \|$.
   Combining the bounds for $\|\ch(P_{k})\|$ and $\| \ch(Q_{k}) \|$, we get the support property on the rational Chen-Ruan cohomology as desired.
   We would like to illustrate the idea by the following two examples.

   \begin{example}\label{exampleQ}
       Let $E\in \cA$ be a $\sigma$-semistable object. For simplicity, we assume $E\in \Coh(\uX)$ and $B=0$.
       By Proposition \ref{chlink}, the support property on rational Chen-Ruan cohomology is equivalent to a good bound on $\|\ch(E_{k})\|$ for every $0\leq k \leq n-1$. Since $\ch_{0}(E_{k})=\ch_{0}(E)$, they are bounded by the support property on the ordinary cohomology. 
       By Proposition \ref{ranks}, $\ch_{1}(E_{k})$ are close to $\ch_{1}(E)$, hence we do not worry about them either. The only difficulty is to bound $\ch_{2}(E_{k})$. 

       For simplicity, we assume $E$ is $\mu_{H}$-stable for a moment. Then by Proposition \ref{ranks}, $E_{k}$ are ``almost" slope stable. A natural idea to bound $\ch_{2}$ is the Bogomolov Inequality. We have the following worst case
       $$\ch_{2}(E_{k}) \sim \dfrac{\ch_{1}(E_{k})^{2}}{\ch_{0}(E_{k})}\sim \mu_{H}(E)^{2}/\ch_{0}(E).$$
       Note that this bound is not linear.
       Hence we see that if $\mu_{H}(E)\gg 0$, this bound is not satisfactory. However, when $\mu_{H}(E)$ is bounded, we get an upper bound of $\ch_{2}(E_{k})$ for \emph{every} $0\leq k \leq n-1$. By Proposition \ref{chlink}, $\ch_{2}(E)$ is approximately the average of $\ch_{2}(E_{k})$, hence we will also get a lower bound for every $\ch_{2}(E_{k})$. This example leads to the object $Q$ in Notation \ref{cases}. 
   \end{example}

   \begin{example}\label{exampleP}
   Let $E\in \cA$ be a $\sigma$-semistable object. For simplicity assume $E\in \Coh(\uX)$ and $B=0$. As the previous Example \ref{exampleQ} points out, what if $\mu_{H}(E) \gg 0$? Note that such $\mu_{H}$-semistable sheaves do exists, and we can make the ratio $\dfrac{\| \ch_{2}(E_{k})\|}{\| \ch_{2}(E) \|}$ as large as possible, as exhibited in the following. 

   Let $p_{1}, \cdots, p_{m}\in C$ be $m$ distinct points and $Z=\{p_{1}, \cdots, p_{m}\}.$ Then we have the following exact sequence 
   $$
   0 \too I_{Z\rho_{0}}(dH) \too \OO_{\uX}(dH) \too \OO_{Z}\rho_{0} \too 0.
   $$
   Let $E=I_{Z\rho_{0}}(dH)$, then approximately
   $\ch_{2}(E)\sim \dfrac{d^{2}H^{2}}{2}-\dfrac{m}{n}$. Hence we may first take arbitrarily large $d\gg 0$ and then arrange $m$ so that $\ch_{2}(E)\sim 0$. However, applying $f_{*}$ to the sequence we see that $E_{0}\cong I_{Z}(dH)$ and $E_{k}\cong \OO_{X}(dH)$ for $1\leq k \leq n-1$. Hence the ratio $\dfrac{\| \ch_{2}(E_{k})\|}{\| \ch_{2}(E) \|}$ is unbounded. 

   The crucial observation here is that such $E$ is not $\sigma$-stable. Indeed, we have the following inclusion
   $$\OO_{\uX}(dH-\sC) \hookrightarrow I_{Z\rho_{0}}(dH).$$
   Observe that the tilt stability has central charge given by (\ref{Zt}):
   	\begin{equation*}
			Z_{t}(\vv)=-\ch_{2}(\vv) + t\ch_{0}(\vv)
			+ i H \cdot (\ch_{1}(\vv)). 
		\end{equation*}
  When $d \gg 0$ and $\ch_{2}(E)$ is bounded, this inclusion would destabilize $I_{Z\rho_{0}}(dH)$. Hence in this case such $E$ are actually not $\sigma$-stable. This example leads to the object $P$ in Notation \ref{cases}. 
   \end{example}

   \begin{remark}
       The readers may wonder why Example \ref{exampleQ} is necessary to consider. The reason why Example \ref{exampleP} cannot unify the proof is that the inclusion 
          $$\OO_{\uX}(dH-\sC) \hookrightarrow I_{Z\rho_{0}}(dH)$$
       may not happen in $\cA$. For instance, when $d$ is bounded and the curve $C$ has sufficient large degree, $\OO_{\uX}(dH-\sC)$ would have negative slope and hence not being an object in $\cA$. Hence to use the idea of Example \ref{exampleP}, we need $\mu_{\min}$ to be not too small. The threshold to distinguish these two types only depends on $\uX$, which will be defined in the following Notation \ref{cases}.
   \end{remark}

   	\begin{notation}\label{cases}
    We set the following notations for this paper.
\begin{enumerate}
\item 
Let $F\in \Coh(\uX)$ and $0=F_{0}\subseteq F_{1}...\subseteq F_{m}=F $ be the $\mu_{H}$-\HN filtration of $F$. For any $a\in \RR$, define 
\begin{itemize}
\item
$
  F_{\leq a}: =
    \begin{cases}
      F/F_{k}, & \text{if $\mu_{\min}^{B}(F)\leq a$ and $k=\min(l: \mu^{B}(F_{l+1}/F_{l})\leq a)$,}\\
      0, & \text{otherwise.}
    \end{cases}       
$
\item
$
  F_{\geq a}: =
    \begin{cases}
      F_{k}, & \text{if $\mu_{\max}^{B}(F)\geq a$ and $k=\max(l: \mu^{B}(F_{l}/F_{l-1})\geq a)$,}\\
      0, & \text{otherwise.}
    \end{cases}       
$
\end{itemize}
Define $F_{<a}$ and $F_{>a}$ similarly. 
\item 
For every $E\in \cA$, define
  \begin{itemize}
      \item $Q= \cH^{0}(E)_{\leq 2\alpha}$ and 
			$P=\ker(E \twoheadrightarrow Q)$. 
   \item $S=\cH^{-1}(E)_{\geq -2\alpha}[1]$ and $T= \mathrm{coker}(S \hookrightarrow E)$. 
  \end{itemize}
Note that by Corollary \ref{world1cor} we have $f^{*}(P_{k}\rho_{k})\in \cA$ and $f^{!}(T_{k}\rho_{k})\in \cA$.
\item
We set the constant
$$\alpha =(n-1)H\sC=\dfrac{n-1}{n}HC. $$ 
\end{enumerate}
		\end{notation}

As stated in the strategy, we first note that the support property holds for the ordinary cohomology in the following Theorem \ref{supp}. Then we will bound $\| \ch_{orb}(Q) \|$ in subsection \ref{boundQ}	and $\| \ch_{orb}(P) \|$ in subsection \ref{boundP}.	
To simplify the computations later we observe the following equivalence of norms.
		\begin{lemma}\label{normequivalence}
			Let $E \in K(\uX)$ and $\| E \|:=\| \ch_{orb}(E)\|$. Using the notaions in Proposition \ref{link}, for every fixed $B\in \mathrm{NS}^{1}(\uX)_{\RR}$, Define 
			$$\| E \|_{B}:=\| (\ch^{B}(E), \ch^{B}(G_{0,1}), \cdots, \ch^{B}(G_{0,n-1} ))\|. $$ 
			Then we have an equivalence of norms $\| \cdot \| \sim \| \cdot \|_{B}. $
		\end{lemma}
		\begin{proof}
			It suffices to show $\| \ch^{B}(G_{0,k}) \| \sim \| \ch(C, G_{0,k}) \|$. The claim follows from Lemma \ref{uch2} and a direct computation.
		\end{proof}

\subsection{Preparations} In this subsection, we collect some necessary facts and lemmas. First, we note that the support property holds for the ordinary cohomology.

		\begin{theorem}\label{supp}
			For every $t>0$, the pre-stability condition $\sigma_{t}$ satisfies the support property for the lattice $\Lambda'$ (Definition \ref{lattice}). 
		\end{theorem}
		\begin{proof}
  The proof is the same as that for smooth surfaces (see e.g. \cite{MS17}).
		\end{proof}

\begin{lemma}\label{Cauchy}
			Let $V$ be an $n$-dimensional real vector space with a positive definite inner product $(- \cdot -)$. Then for every $\vv_{j}\in V, a_{j}>0, 1\leq j \leq m$, we have
			$$
			(\sum_{j=1}^{m} a_{j}\vv_{j})^{2}\leq  \sum_{j=1}^{m}a_{j}^{2}\sum_{j=1}^{m}\vv_{j}^{2}.
			$$
		\end{lemma}
		
		\begin{proof}
		Let $e_{1}, \cdots e_{n}\in V$ be an orthogonal basis of $V$. Write $\vv_{j}=\sum_{i=1}^{n} x^{j}_{i}e_{i}, x^{j}_{i}\in \RR$. By the Cauchy-Schwarz inequality, we have
		\begin{eqnarray*}
		(\sum_{j=1}^{m} a_{j}\vv_{j})^{2}
		&=&(\sum_{i=1}^{n} \sum_{j=1}^{m} a_{j}x^{j}_{i}e_{i})^{2}=
		\sum_{i=1}^{n} (\sum_{j=1}^{m}a_{j} x^{j}_{i}e_{i})^{2}\\
		&\leq &\sum_{i=1}^{n}\sum_{j=1}^{m}a_{j}^{2}\sum_{k=1}^{m}(x_{i}^{k}e_{i})^{2}
		=\sum_{j=1}^{m}a_{j}^{2}\sum_{k=1}^{m}(\sum_{i=1}^{n}x_{i}^{k}e_{i})^{2}=\sum_{j=1}^{m}a_{j}^{2}\sum_{k=1}^{m}\vv_{k}^{2}.
	\end{eqnarray*}
		\end{proof}
		
		\begin{proposition}
			Let $F\in \Coh(\uX)$ be a torsion free sheaf such that $a\leq \mu^{B}_{\min}(F) \leq \mu^{B}_{\max}(F) \leq b$. Then we have 
			$$ \ch_{2}^{B}(F)\leq \dfrac{\max(a^{2}, b^{2})}{2H^{2}}\cdot \ch_{0}^{B}(F). $$
		\end{proposition}
		
		\begin{proof}
			We prove a stronger inequality that will be used later. Let $G_{1}, \cdots, G_{m}$ be the $\mu_{H}$-\HN factors of $F$. Write $\ch_{1}^{B}(G_{j})=a_{j}H+N_{j}$, where $a_{j}=\dfrac{H\ch_{1}^{B}(G_{j})}{H^{2}}$, $N_{j}\in H^{\perp}$. Let $a=\sum_{j=1}^{m}a_{j}$ and $N=\sum_{j=1}^{m}N_{j}$, then $\ch_{1}^{B}(F)=aH+N$. By the Bogomolov inequality, for every $G_{j}$, we have 
			$$
			\ch_{2}^{B}(G_{j})\leq \dfrac{\ch_{1}^{B}(G_{j})^{2}}{2\ch_{0}^{B}(G_{j})}=\dfrac{a_{j}^{2}H^{2}+N_{j}^{2}}{2\ch_{0}^{B}(G_{j})}
			=\dfrac{\mu^{B}(G_{j})}{2}a_{j}+\dfrac{N_{j}^{2}}{2\ch_{0}^{B}(G_{j})}. 
			$$
			By Lemma \ref{Cauchy} and the Hodge Index Theorem, we have
			\begin{equation*}
				(\sum_{j=1}^{m}N_{j})^{2}=\left(\sum_{j=1}^{m} \sqrt{\ch_{0}^{B}(G_{j})} \dfrac{N_{j}}{\sqrt{\ch_{0}^{B}(G_{j})}}\right)
				\geq  \sum_{j=1}^{m}\ch_{0}^{B}(G_{j}) \sum_{j=1}^{m} \dfrac{N_{j}^{2}}{\ch_{0}^{B}(G_{j})}
				= \ch_{0}^{B}(F) \sum_{j=1}^{m} \dfrac{N_{j}^{2}}{\ch_{0}^{B}(G_{j})}.
			\end{equation*}
			Hence we have 
			\begin{align}
				\ch_{2}^{B}(F)&=\sum_{j=1}^{m}\ch_{2}^{B}(G_{j}) \leq \sum_{j=1}^{m}\dfrac{\mu^{B}(G_{j})}{2}a_{j}+\sum_{j=1}^{m}\dfrac{N_{j}^{2}}{2\ch_{0}^{B}(G_{j})}\nonumber \\
				&\leq \sum_{j=1}^{m}\dfrac{\mu^{B}(G_{j})}{2}a_{j}+ \dfrac{(\sum_{j=1}^{m}N_{j})^{2}}{2  \ch_{0}^{B}(F)}
				=\sum_{j=1}^{m}\dfrac{\mu^{B}(G_{j})^{2}}{2H^{2}}\ch_{0}^{B}(G_{j})+\dfrac{N^{2}}{2  \ch_{0}^{B}(F)}\nonumber \\
				 & \leq \dfrac{\max(a^{2}, b^{2})}{2H^{2}}\sum_{j=1}^{m}\ch_{0}^{B}(G_{j})+ \dfrac{N^{2}}{2  \ch_{0}^{B}(F)}
				= \dfrac{\max(a^{2}, b^{2})}{2H^{2}}\ch_{0}^{B}(F) + \dfrac{N^{2}}{2  \ch_{0}^{B}(F)}.
				\label{BOG}
			\end{align}
			By the Hodge Index Theorem, $N^{2}\leq 0$. Hence
			$$
			\ch_{2}^{B}(F)\leq 
			\dfrac{\max(a^{2}, b^{2})}{2H^{2}}\ch_{0}^{B}(F) + \dfrac{N^{2}}{2  \ch_{0}^{B}(F)}
			\leq 
			\dfrac{\max(a^{2}, b^{2})}{2H^{2}}\ch_{0}^{B}(F).
			$$
		\end{proof}

		\begin{lemma}\label{triangle}
			Let $\mathcal{D}$ be a triangulated category, and $\cA$ be the heart of a bounded $t$-structure. Suppose $A, B, C\in \cA$, then 
		 $A \to B \to C $ is a short exact sequence in $\cA$ if and only if it is a distinguished triangle in $\mathcal{D}$.
		\end{lemma}
		
		\begin{proof}
              Now suppose  $A\to B\to C\to A[1]$ is a distinguished triangle.
			There are cohomology object functors $\cH_{\mathcal{A}}^{\bullet}: \mathcal{D} \to \mathcal{A} $. 
   Taking the cohomology, we get a short exact sequence $$ 0\to A \to B \to C \to 0. $$
   
            Conversely, let $0\to A \to B \to C \to 0 $ be a short exact sequence in $\cA$. Let $C'$ be the cone of $A\to B$, then we have the following exact sequence: 
             $$\operatorname{Hom}(C', C)\to \operatorname{Hom}(B,C)\to \operatorname{Hom}(A,C).$$
             Since the composition $A\to B \to C$ is 0, the map $B\to C$ factors through $C'$. Hence we have a commutative diagram: 
             \begin{center}
\begin{tikzcd}
A \arrow[r] \arrow[d, "\id"] & B \arrow[r] \arrow[d, "\id"] & C' \arrow[r] \arrow[d] & {A[1]} \\
A \arrow[r]                 & B \arrow[r]                 & C.                      &       
\end{tikzcd}

          \end{center}
Now applying the cohomology functor, we get a commutative diagram: 
\begin{center}

\begin{tikzcd}
0 \arrow[r] & \cH_{\cA}^{-1}(C') \arrow[r] & A \arrow[r] \arrow[d, "\id"] & B \arrow[r] \arrow[d, "\id"] & \cH_{\cA}^{0}(C') \arrow[r] & 0 \\
            & 0 \arrow[r] & A \arrow[r]                 & B \arrow[r]                 & C \arrow[r] & 0.
\end{tikzcd}
     \end{center}     
Hence we have $\cH_{\cA}^{-1}(C')=0$ and $\cH_{\cA}^{0}(C')=C$. Since $\cH^{i}_{\cA}(A)=\cH^{i}_{\cA}(B)=0$ for $i\neq 0$, we also know $\cH_{\cA}^{i}(C')=0$ for all $i\not = -1,0$. Hence $C'\cong C$ and we have a distinguished triangle: $$A\to B \to C\to A[1].$$

		\end{proof}

Next we observe the fact that the first Chern classes of $E_{k}$ are close to $E$. This will be used often later.
  
	\begin{proposition}\label{world1}
			Let $E\in \Coh(\uX)$ be a torsion free sheaf. Then for every $0 \leq k \leq n-1$, we have
			$$ \mu^{B}_{\min}(E)- \alpha \leq \mu^{B}_{\min}(f^{*}(E_{k}\rho_{k})) \leq  \mu^{B}_{\max}(f^{*}(E_{k}\rho_{k})) \leq \mu^{B}_{\max}(E). $$
			In particular, if $\mu^{B}_{\min}(E)> (n-1)H\sC$, then for every $0\leq k \leq n-1$ we have $f^{*}(E_{k}\rho_{k})\in \cA$.
		\end{proposition}
		
		\begin{proof}
			Recall we have map $f:\uX\to \mathscr{X}$. Suppose $f_{*}(E)=\bigoplus_{k=0}^{n-1} E_{k}\rho_{k}$.  The canonical projection $ f_{*}(E) \to E_{k}\rho_{k}$ induces a map 
			$$ p_{k}: E \to f^{!}(E_{k}\rho_{k})=E_{k}((n-1-k)\sC). $$
			Let $E_{k}((n-1-k)\sC) \to Q$ be the last $\mu^{B}$-\HN factor. Then $Q$ is torsion free. Since $p_{k}$ is an isomorphism away from $C$, in particular the composition map $E \to Q$ is nonzero. Since $Q$ is $\mu^{B}$-semistable, we have $\mu^{B}_{\min}(E)\leq \mu^{B}(Q)$. 
			By assumption, we have
			$$ \mu^{B}_{\min}(E)\leq \mu^{B}(Q)=\mu^{B}_{\min}(E_{k}((n-1-k)\sC)) =\mu^{B}_{\min}(E_{k})+(n-1-k)H\sC. $$
			Hence $f^{*}(E_{k}\rho_{k})\in \cA$, since we have
			$$ \mu^{B}_{\min}(f^{*}(E_{k}\rho_{k}))=\mu^{B}_{\min}(E_{k}(-k\sC))=\mu^{B}_{\min}(E_{k})-kH\sC\geq \mu^{B}_{\min}(E)- (n-1)H\sC. $$
			To see the other inequality, note that the inclusion $E_{k}\rho_{k} \to f_{*}(E)$ induces a map
			$f^{*}(E_{k}\rho_{k}) \to E$. The map is injective since it is an isomorphism away from $C$ and $E$ is torsion free. Hence we have 
			$$\mu^{B}_{\max}(f^{*}(E_{k}\rho_{k})) \leq \mu^{B}_{\max}(E).$$
		\end{proof}

		\begin{corollary}\label{world1cor}
			For every $E\in \cA$, we have
			\begin{itemize}
				\item If $\mu^{B}_{\min}(\cH^{0}(E))>\alpha $, then $f^{*}(E_{k}\rho_{k})\in \cA$ for $0\leq k \leq n-1$. 
				\item If $\mu^{B}_{\max}(\cH^{-1}(E))\leq -\alpha $, then $f^{!}(E_{k}\rho_{k})\in \cA$ for $0\leq k \leq n-1$.
			\end{itemize}
		\end{corollary}
		
		\begin{proof}
			The first claim is Proposition \ref{world1}. To prove the second claim, note that $f^{*}$ and $f_{*}$ are exact. Hence $\mathcal{H}^{i}(f^{!}(E_{k}\rho_{k}))=f^{!}(\mathcal{H}^{i}(E_{k}\rho_{k}))=f^{!}(\mathcal{H}^{i}(E)_{k}\rho_{k})$, the claim follows by applying Proposition \ref{world1} to $\mathcal{H}^{-1}(E)$. 
		\end{proof}

	\subsection{Bounding $\| \ch_{orb}(Q) \|$}\label{boundQ}
The goal of this subsection is to prove the following proposition. 

\begin{proposition}\label{mainQ}
    Fix $t>0$ and let $E\in \cA$ be a $\sigma_{t}$-semistable object. Using Notation \ref{cases}, there is a constant $M_{Q}$ such that 
    $$
    \| \ch^{B}(Q_{k}) \|\leq M_{Q}|Z(E)|, \quad 0 \leq k \leq n-1.
    $$
    Similarly, there is a constant $M_{S}$ such that 
    $$
    \| \ch^{B}(S_{k}) \|\leq M_{S}|Z(E)|, \quad 0 \leq k \leq n-1.
    $$
\end{proposition}

The following proposition is a necessary condition for Proposition \ref{mainQ} to be true. It will be used later.
 
\begin{proposition}\label{Q}
		Fix $t>0$ and let $E\in \cA$ be a $\sigma_{t}$-semistable object. Using Notation \ref{cases}, there are constants $M_{1}, M_{2}$, such that 
		$$
		|\ch_{2}^{B}(Q)|\leq M_{1}|Z(E)|, \quad 0\leq \ch_{0}^{B}(Q)\leq M_{2}|Z(E)|.
		$$
		Similarly, there are constants $M_{1}', M_{2}'$, such that 
		$$
		|\ch_{2}^{B}(S)|\leq M_{1}'|Z(E)|, \quad M_{2}' |Z(E)| \leq \ch_{0}^{B}(S)\leq 0.
		$$
	\end{proposition}
	
	\begin{proof}
		We prove the first claim, the second claim is proved similarly. 
		First note that $\ch_{0}^{B}(Q)>0$, otherwise $Q=0$ and there is nothing to prove. Let $0<a<\sqrt{2tH^{2}}$ to be any constant. Let $Q'=Q_{\geq a}$ and $Q''=Q_{< a}$. We first deal with $Q''$.
		\\
		\textbf{Bounding $Q''$:}
		 In the distinguished triangle
		\begin{equation}\label{triangle1}
	F \to E \to Q'' \to F[1],
		\end{equation}  
		we have $\cal{H}^{-1}(F)=\cal{H}^{-1}(E)$ and $\cal{H}^{0}(F)=S$. Hence $F\in \cA$. By Lemma \ref{triangle}, (\ref{triangle1}) is an exact seqeunce in $\cA$. By $\sigma_{t}$-semistability of $E$, we have $\mu_{\sigma}(E) \leq \mu_{\sigma}(Q'')$. Hence we have
		\begin{equation*}
			-\ch_{2}^{B}(Q'')+t\ch_{0}^{B}(Q'') \leq \dfrac{H\ch_{1}^{B}(Q'')}{H\ch_{1}^{B}(E)}(-\ch_{2}^{B}(E)+t\ch_{0}^{B}(E)).
		\end{equation*}
		By Proposition \ref{BOG}, we have $\ch_{2}^{B}(Q'')\leq \dfrac{a^{2}}{2H^{2}}\ch_{0}^{B}(Q'')$. Hence we have 
		\begin{equation}\label{ineq1}
			0< \left(t-\dfrac{a^{2}}{2H^{2}}\right)\ch_{0}^{B}(Q'') \leq 
			\dfrac{H\ch_{1}^{B}(Q'')}{H\ch_{1}^{B}(E)}(-\ch_{2}^{B}(E)+t\ch_{0}^{B}(E)) \leq 
			\Re(Z(E)).
		\end{equation}
		Since $t-\dfrac{a^{2}}{2 H^{2}}>0$, we may take $a_{2}=\left(t-\dfrac{a^{2}}{2H^{2}}\right)^{-1}$. Then $0\leq \ch_{0}^{B}(Q'')\leq a_{2}|Z(E)|$.
		
		By (\ref{ineq1}), if $Q''\neq 0$ then $\Re(Z(E))> 0$. Hence we have
		\begin{equation*}
			-\ch_{2}^{B}(Q'')<-\ch_{2}^{B}(Q'')+t\ch_{0}^{B}(Q'') 
			\leq 
			\dfrac{H\ch_{1}^{B}(Q'')}{H\ch_{1}^{B}(E)}(-\ch_{2}^{B}(E)+t\ch_{0}^{B}(E)) \leq 
			\Re(Z(E)).
		\end{equation*}
		By Proposition \ref{BOG}, we have 
		\begin{equation*}
		\ch_{2}^{B}(Q'')\leq \dfrac{a^{2}}{2H^{2}}\ch_{0}^{B}(Q'') \leq \dfrac{a^{2}}{2H^{2}}\cdot a_{2} |\Re(Z(E))|.
		\end{equation*}
        Hence we may take $a_{1}=\max\left(\dfrac{2\alpha^{2}}{H^{2}}\cdot a_{2}, 1\right). $ Then $\|\ch_{2}^{B}(Q'')\|\leq a_{1}|Z(E)|$.
        \\
        \textbf{Bounding $Q$:}
        Since $\mu^{B}(Q') \geq a>0$, we have $\ch_{0}^{B}(Q')\leq a^{-1}H\ch_{1}^{B}(Q') \leq a^{-1}|Z(E)|$. Let $M_{2}=a_{2}+a^{-1}$, then we have
        $$ 
        0\leq \ch_{0}^{B}(Q)\leq a_{2}|Z(E)| + a^{-1}|Z(E)|=M_{2}|Z(E)|.
        $$
        Since $0< \mu_{\min}^{B}(Q)\leq \mu_{\max}^{B}(Q) \leq 2\alpha$, we have $\ch_{2}^{B}(Q)\leq \dfrac{2\alpha^{2}}{H^{2}}\ch_{0}^{B}(Q)\leq \dfrac{2\alpha^{2}}{H^{2}}M_{2}|Z(E)|.$
        Since $E \twoheadrightarrow Q$, by stability of $E$ we have $\mu_{\sigma}(E)\leq \mu_{\sigma}(Q)$. Hence 
        $$
        			-\ch_{2}^{B}(Q)\leq -\ch_{2}^{B}(Q)+t\ch_{0}^{B}(Q) \leq \dfrac{H\ch_{1}^{B}(Q)}{H\ch_{1}^{B}(E)}(-\ch_{2}^{B}(E)+t\ch_{0}^{B}(E))\leq 
        			\dfrac{H\ch_{1}^{B}(Q)}{H\ch_{1}^{B}(E)}|Z(E)|\leq |Z(E)|.
        $$
        
	\end{proof}
	
In order to bound $\| \ch_{1}^{B}(Q_{k}) \|$, we need the following lemmas. 

\begin{lemma}\label{errorQ}
	Under the conditions of Proposition \ref{Q}, there is a constant $M_{3}$, such that 
	$$ \|\ch_{1}^{B}(Q)\|\leq M_{3} |Z(E)|. $$
	Similarly, there is a constant $M_{3}' $, such that 
	$$ \|\ch_{1}^{B}(S)\|\leq M_{3}' |Z(E)|. $$
\end{lemma}

\begin{proof}
	We prove the first claim, the second claim is proved similarly. 
	Let $\gr_{1}(Q), \cdots,\gr_{m}(Q) $ be the $\mu^{B}$-\HN factors of $Q$. Then we may write $$\ch_{1}^{B}(\gr_{j}(Q))=a_{j}H+N_{j}, a_{j}= \dfrac{H\ch_{1}^{B}(\gr_{j}(Q))}{H^{2}}, N_{j}\in H^{\perp}. $$
	Let $a=\sum_{j=1}^{m}a_{j}=\dfrac{H\ch_{1}^{B}(Q)}{H^{2}}$ and $N=\sum_{j=1}^{m} N_{j}$. Then $N\in H^{\perp}$, and $\ch_{1}^{B}(Q)=aH+N.$
	By a modified version of Bogomolov inequality (\ref{BOG}), we have
	\begin{equation}
		N^{2}\geq 2 \ch_{2}^{B}(Q)\ch_{0}^{B}(Q)-\dfrac{4\alpha^{2}}{H^{2}}\ch_{0}^{B}(Q)^{2}.
	\end{equation}
	By Proposition \ref{Q}, there is a constant $b_{1}>0$, such that 
	\begin{equation}
		N^{2}\geq -b_{1}|Z(E)|^{2}.
	\end{equation} 
	Since $N\in H^{\perp}$ and by the Hodge Index Theorem $H^{\perp}$ is negative definite, there is a constant $b_{2}>0$ such that 
	$\|N\| \leq b_{2}|Z(E)|$. 
	Note that $Q$ is a quotient of $E$ in $\cA$, 
	$$ a=\dfrac{H\ch_{1}^{B}(Q)}{H^{2}} \leq \dfrac{H\ch_{1}^{B}(E)}{H^{2}}= \dfrac{\Im(Z(E))}{H^{2}}.$$
	Hence we have 
	$$
	\|\ch_{1}^{B}(Q)\|=\|aH+N\|\leq \|aH\|+\|N\|\leq \dfrac{\| H \|}{H^{2}}|\Im(Z(E))|+b_{2}|Z(E)|.
	$$
	
\end{proof}

	\begin{lemma}\label{c1close2}
			There exists a constant $b>0$, such that for every torsion free sheaf $E\in \Coh(\uX)$, we have
			$$\ch_{1}^{B}(E_{k})C\leq \ch_{1}^{B}(E)C +b\ch_{0}^{B}(E). $$
		\end{lemma}
		
		\begin{proof}
			By Lemma \ref{chlink}, we have
			\begin{align*}
				\ch_{1}^{B}(E_{k})C =\ch_{1}^{B}(E)C-\ch_{1}^{B}(E)C+\ch_{1}^{B}(E_{k})C
                    &=\ch_{1}^{B}(E)C+ \dfrac{1}{n}\sum_{j=0}^{n-1}(\ch_{1}(E_{k}) -\ch_{1}(E_{j}))C \\
				&=\ch_{1}^{B}(E)C+ \dfrac{1}{n} \sum_{0\leq j \neq k \leq n-1}\ch_{1}(G_{jk})C.
                         \end{align*}
			By Lemma \ref{ranks}, we have 
   $$\ch_{1}^{B}(E)C+\dfrac{1}{n} \sum_{0\leq j \neq k \leq n-1}\ch_{0}(C, G_{jk})C^{2} 
                \leq \ch_{1}^{B}(E)C+ \dfrac{n-1}{n}|C^{2}|\ch_{0}^{B}(E)$$
		\end{proof}  

  Now we are ready to bound $\| \ch_{1}^{B}(Q_{k}) \|$.

\begin{proposition}\label{solutionQ}
	Under the condition of Proposition \ref{Q}, write $f_{*}(Q)=\bigoplus_{k=0}^{n-1}Q_{k}\rho_{k}$ and $f_{*}(S)=\bigoplus_{k=0}^{n-1}S_{k}\rho_{k}$. Then there are constants $M_{4}, M_{5}$, such that
	$$\|\ch_{2}^{B}(Q_{k})\|\leq M_{4}|Z(E)|, \|\ch_{1}^{B}(Q_{k})\|\leq M_{5}|Z(E)| \quad 0\leq k \leq n-1. $$
	Similarly, there is a constant $M_{4}'$, such that 
	$$\|\ch_{2}^{B}(S_{k})\|\leq M_{4}' |Z(E)|, \|\ch_{1}^{B}(S_{k})\|\leq M_{5}'|Z(E)| \quad 0\leq k \leq n-1. $$
\end{proposition}

\begin{proof}
	We prove the first claim, the second claim is proved similarly. 
	\\
	\textbf{Bounding $\ch_{2}(Q_{k})$:}
	
	By Proposition \ref{world1}, we have 
	\begin{equation*}
		-\alpha\leq \mu_{\min}^{B}(Q)-\alpha \leq \mu_{\min}^{B}(Q_{k}) \leq \mu_{\max}^{B}(Q_{k}) \leq \mu_{\max}^{B}(Q)+\alpha \leq 3\alpha. 
	\end{equation*}
	Hence by Proposition \ref{BOG} and Proposition \ref{Q}, we have 
	\begin{equation}\label{equation 5.11}
		\ch_{2}^{B}(Q_{k})\leq \dfrac{9\alpha^{2}}{2H^{2}} \ch_{0}^{B}(Q_{k})=
		\dfrac{9\alpha^{2}}{2H^{2}} \ch_{0}^{B}(Q)\leq \dfrac{9\alpha^{2}}{2H^{2}} M_{1}|Z(E)|, \quad 0\leq k \leq n-1. 
	\end{equation}
	By Proposition \ref{chlink}, we have 
	\begin{equation*}
	\ch_{2}^{B}(Q) =\dfrac{1}{n}\sum_{k=0}^{n-1} \ch_{2}^{B}(Q_{k}) + \sum_{k=0}^{n-1} \dfrac{n-2k-1}{2n^{2}}\ch_{1}^{B}(Q_{k})\cdot C. 
	\end{equation*}
	By Lemma \ref{c1close2}, there are constants $B_{0},...,B_{k}$  
         \begin{align*} 
\ch_{2}^{B}(Q) & \leq\dfrac{1}{n}\sum_{k=0}^{n-1}\ch_{2}^{B}(Q_{k})+ \sum_{k=0}^{n-1} \dfrac{n-2k-1}{2n^{2}}(\ch_{1}^{B}(Q)C+B_{k}\ch_{0}^{B}(Q)) \\ 
           &=  \dfrac{1}{n}\sum_{k=0}^{n-1}\ch_{2}^{B}(Q_{k}) +\sum_{k=0}^{n-1} \dfrac{n-2k-1}{2n^{2}}B_{k}\ch_{0}^{B}(Q).
        \end{align*} 
 Thus there is a constant $B'>0$, such that 
	\begin{equation*}
		\ch_{2}^{B}(Q)\leq \dfrac{1}{n}\sum_{k=0}^{n-1}\ch_{2}^{B}(Q_{k})+B'\ch_{0}^{B}(Q).
	\end{equation*}
	By Proposition \ref{chlink}, Proposition \ref{Q} and (\ref{equation 5.11}), there is a constant $b_{3}>0$, such that 
	\begin{equation*}
		\ch_{2}^{B}(Q_{j})\geq n\ch_{2}^{B}(Q)-\sum_{k\neq j}\ch_{2}^{B}(Q_{k})-B'\ch_{0}^{B}(Q)\geq -b_{3}|Z(E)|. 
	\end{equation*}
	Hence we may take $M_{4}=\max(\dfrac{9\alpha^{2}}{2H^{2}} M_{1}, b_{3})$.
	\\
	\textbf{Bounding $\ch_{1}(Q_{k})$:}
	
	By Proposition \ref{link}, we may write $\ch_{1}^{B}(Q_{k})=\ch_{1}^{B}(Q)+b\sC$ for some $b\in \ZZ$. By Lemma \ref{errorQ}, it suffices to bound $b$. By Proposition \ref{world1}, we have 
	$$-\alpha<\mu_{\min}^{B}(Q)-\alpha\leq \mu^{B}(Q_{k})=\mu^{B}(Q)+b\dfrac{H\sC}{\ch_{0}^{B}(Q)} \leq \mu_{\max}^{B}(Q)+\alpha \leq 3\alpha. $$
	Hence by Proposition \ref{Q}, we have
	$$
	|b| \leq \dfrac{3\alpha}{H\sC} \ch_{0}^{B}(Q) \leq \dfrac{3\alpha}{H\sC}M_{2}|Z(E)|.
	$$

 \begin{proof}[Proof of Proposition \ref{mainQ}]
     We prove the first claim, the second claim is similar.
     Note that $\ch_{0}(Q_{k})=\ch_{0}(Q)$, by Proposition \ref{Q} and Theorem \ref{supp} it is bounded. The proposition now follows from Proposition \ref{solutionQ}.
 \end{proof}

\end{proof}

\subsection{Bounding $\| \ch_{orb}(P) \|$}\label{boundP}

By Proposition \ref{mainQ}, if we have a bound for $\| \ch^{B}(P_{k}) \|$, then we will get a bound for $\| \ch^{B}(E_{k}) \|$. The goal of this subsection is the following proposition. 

\begin{proposition}\label{mainP}
    Fix $t>0$ and let $E\in \cA$ be a $\sigma_{t}$-semistable object. Using Notation \ref{cases}, there is a constant $M_{E}$ such that 
    $$
    \| \ch^{B}(E_{k}) \|\leq M_{E}|Z(E)|, \quad 0 \leq k \leq n-1.
    $$
\end{proposition}

We first prove some lemmas.
\begin{lemma}\label{zerocomponent}
	Let $0 \to E \to F \to G \to 0$ be a short exact sequence in $\cA$. For every $0\leq k \leq n-1$, if the induced map $E_{k} \to F_{k}$ on $\rho_{k}$ component is zero, then $E\in \Coh(\uX)$. Furthermore, either 
	\begin{enumerate}
		\item $E_{k}=0$, or
		\item $E_{k}$ is torsion free with $\mu_{\max}^{B}(E_{k}) \leq \mu_{\max}^{B}(\cH^{-1}(G_{k}))$.  
	\end{enumerate}
\end{lemma}

\begin{proof}
	We prove the lemma for $k=0$, the other cases are similar. 
  By Lemma \ref{triangle}, we have a distinguished triangle in $D^{b}(\uX)$: $$E\to F \to G \to E[1].$$
Since $f_{*}$ is exact, applying the functor $f_{*}=Rf_{*}$, we have a distinguished triangle in $D^{b}(\mathscr{X})$:
$$f_{*}E\to f_{*}F \to f_{*} G \to f_{*}E[1].$$
Hence we get a distinguished triangle in $D^{b}(X)$:
$$E_{0}\to F_{0} \to G_{0}\to E_{0}[1].$$

 Consider the long exact sequence of cohomology sheaves 
	\begin{equation*}
		0 \to \cH^{-1}(E_{0}) \to \cH^{-1}(F_{0}) \to \cH^{-1}(G_{0}) \to \cH^{0}(E_{0}) \to \cH^{0}(F_{0}).
	\end{equation*}
 
	Since $\cH^{-1}(E_{0}) \to \cH^{-1}(F_{0})$ is an injective zero map, we have $\cH^{-1}(E_{0})=0$ and $\cH^{0}(E_{0})=E_{0}$. Since $\cH^{-1}(E)$ is torsion free and $\cH^{-1}(E_{0})=0$, we have $\cH^{-1}(E)=0$ and $E\in \Coh(\uX)$. 
 
Note $G_{0}\cong \operatorname{Cone}(E_{0}\to F_{0})=E_{0}[1]\oplus F_{0}$ in $D^{b}(X)$. Taking the cohomology,  we have 
	$\cH^{-1}(G_{0})\cong \cH^{-1}(F_{0}) \oplus E_{0}$. Note that $\cH^{-1}(G_{0})$ is torsion free, hence either $E_{0}=0$, or $E_{0}$ is torsion free and 
	$ \mu_{\max}^{B}(E_{0})\leq \mu_{\max}^{B}(\cH^{-1}(G_{0})). $
\end{proof}

\begin{lemma}\label{trouble1}
	Using notation \ref{cases}, fix any $0\leq k \leq n-1$ and let $A$ be the first $\sigma_{t}$-\HN factor of $f^{*}(P_{k}\rho_{k}).$ If the composed morphism $A \to f^{*}(P_{k}\rho_{k}) \to P$ is zero, then $A\in \Coh(\uX)$, and either
	\begin{enumerate}
		\item $A_{k}=0$, or
		\item For every $0\leq j \leq n-1$, $-\alpha\leq \mu_{\min}^{B}(f^{*}(A_{j}\rho_{j}))\leq \mu_{\max}^{B}(f^{*}(A_{j}\rho_{j})) \leq \alpha$. In addition, $A_{k}$ is torsion free.
	\end{enumerate}
	Similarly, let $A'$ be the last $\sigma_{t}$-\HN factor of $f^{!}(T_{k}\rho_{k}).$ If the composed morphism $P \to f^{!}(T_{k}\rho_{k}) \to A'$ is zero, then $\cH^{0}(A')$ is torsion, and either
	\begin{enumerate}[label=(\alph*)]
		\item $A'_{k}=0$, or
		\item For every $0\leq j \leq n-1$, $-\alpha\leq \mu_{min}^{B}(f^{!}(A'_{j}\rho_{j}))\leq \mu_{\max}^{B}(f^{!}(A'_{j}\rho_{j})) \leq \alpha$. In addition, $\cH^{0}(A_{k})=0$.
	\end{enumerate}
\end{lemma}

\begin{proof}
	We prove the first claim, the second claim is proved similarly. 
	For simplicity, we assume $k=0$, other cases are proved similarly. 
	
	If $E\not\in \Coh(\uX)$, then $\cH^{-1}(P)=\cH^{-1}(E)\neq 0$ is torsion free. Let $R$ be the cokernel of $A \to f^{*}(P_{0}\rho_{0})=\pi^{*}(P_{0})$ in $\cA$. Apply $\pi_{*}$ to the short exact sequence 
	$$
	0 \to A \to \pi^{*}(P_{0}) \to R \to 0, 
	$$
	we have the following exact triangle
	\begin{equation*}
		A_{0} \to P_{0} \to R_{0} \to A_{0}[1].
	\end{equation*}
	Apply $\pi_{*}$ to the composition $A \to f^{*}(P_{0}\rho_{0}) \to P$, we have the composition 
	$ A_{0} \to P_{0} \overset{\id}{\to} P_{0}$. If $A \to P$ is zero, then $A_{0} \to P_{0}$ is zero. 
	By Lemma \ref{zerocomponent}, we have $A\in \Coh(\uX)$. Furthermore, either $A_{0}=0$, or $A_{0}$ is torsion free and 
	$$\mu_{\max}^{B}(A_{0})\leq \mu_{\max}^{B}(\cH^{-1}(R_{0})) \leq \mu_{\max}^{B}(\cH^{-1}(R)) \leq 0. $$
	Since $\mu^{B}(\cH^{0}(A))>0$, by Proposition \ref{world1}, we have 
	$$ -\alpha< \mu^{B}(A_{0})=\mu^{B}(\cH^{0}(A_{0}))\leq 0. $$
	By Proposition \ref{world1}, we have
	\begin{equation}\label{tf}
	0<\mu_{\min}^{B}(A)\leq \mu_{\max}^{B}(A) \leq \mu_{\max}^{B}(A_{0})+\alpha\leq \alpha.
	\end{equation}
	Hence by Proposition \ref{world1}, for every $0\leq j \leq n-1$, we have 
	$$-\alpha\leq \mu_{\min}^{B}(f^{*}(A_{j}\rho_{j}))\leq \mu_{\max}^{B}(f^{*}(A_{j}\rho_{j}))\leq  \alpha. $$
\end{proof}

\begin{corollary}\label{solution2}
	Fix $t>0$. Then there is a constant $M_{6}$ which only depends on $\uX$ and $t$, such that 
	\begin{itemize}
		\item In case (ii) of Lemma \ref{trouble1}, for every $0\leq k \leq n-1$ we have $\mu_{\sigma}(f^{*}(P_{k}\rho_{k}))\leq M_{6}$.
		\item In case (b) of Lemma \ref{trouble1}, for every $0\leq k\leq n-1$ we have $\mu_{\sigma}(f^{!}(T_{k}\rho_{k}))\geq -M'_{6}$. 
	\end{itemize} 
\end{corollary}

\begin{proof}
	We prove the first claim, the second claim is proved similarly. 
	We have a short exact sequence in $\cA$:
	$$
	0 \to A_{\tor} \to A \to A_{\mathrm{tf}} \to 0. 
	$$
	By (\ref{tf}), $A_{\mathrm{tf}}$ is a torsion free sheaf in $\cA$ with $\mu_{\max}(A_{\mathrm{tf}})\leq 2\alpha$. 

 	Fix a constant $0 < a <\sqrt{2tH^{2}}$. Let $F'=(A_{\mathrm{tf}})_{\geq a}, F''=(A_{\mathrm{tf}})_{<a}$. By Proposition \ref{BOG}, we have $\ch_{2}^{B}(F'')\leq \dfrac{a^{2}}{2H^{2}}\ch_{0}^{B}(F'')<t\ch_{0}^{B}(F'')$. Then we have 
	$$
	\mu_{\sigma}(F'')=\dfrac{\ch_{2}^{B}(F'')-t\ch_{0}^{B}(F'')}{H\ch_{1}^{B}(F'')}<0. 
	$$
	Since $a\leq \mu_{\min}(F') \leq \mu_{\max}(F') \leq 2\alpha $, by Proposition \ref{BOG}, we have 
	$\ch_{2}^{B}(F')\leq \dfrac{2\alpha^{2}}{H^{2}}$, and
	$$
	\mu_{\sigma}(F')
	=\dfrac{\ch_{2}^{B}(F')-t\ch_{0}^{B}(F')}{H\ch_{1}^{B}(F')}\leq \left(\dfrac{2\alpha^{2}}{H^{2}}-t\right)\dfrac{\ch_{0}^{B}(F')}{H\ch_{1}^{B}(F')}
	=\left(\dfrac{2\alpha^{2}}{H^{2}}-t\right)\dfrac{1}{\mu^{B}(F')}
	\leq |\dfrac{2\alpha^{2}}{H^{2}}-t | a^{-1}.
	$$
	Hence 
	$\mu_{\sigma}(F)\leq \max(\mu_{\sigma}(F'), \mu_{\sigma}(F'')) \leq \max(0,  |\dfrac{2\alpha^{2}}{H^{2}} -t | a^{-1})$.

 By stability of $A$, there is a constant $M_{6}$ that only depends on $\uX$ and $t$, such that
	$$
	\mu_{\sigma}(f^{*}(P_{k}\rho_{k})) \leq \mu_{\sigma}(A)\leq \mu_{\sigma}(A_{\mathrm{tf}})\leq M_{6}. 
	$$
\end{proof}

Now we deal with case (i) and (a) of Lemma \ref{trouble1}. First we need some lemmas. 

\begin{lemma}\label{minslope}
 Let $F\in \Coh(C)$ and fix $0\leq k \leq n-1$. For every 
 $0\leq j, k \leq n-1$, if $J\in \cA$ is a torsion sheaf such that $F\rho_{j} \twoheadrightarrow J$ attains  $\phi_{\min}(F\rho_{j})$, then we have
 $$
 \mu_{\sigma, \min}(F\rho_{j}) \geq \mu_{\sigma, \min}(F\rho_{k})+\dfrac{(k-j)\sC^{2}}{H\sC}.
 $$
\end{lemma}

\begin{proof}
Let $S$ be the kernel of $F\rho_{j} \twoheadrightarrow J$ in $\cA$. By taking cohomology sheaves, we see that 
$$ 0 \to S \to F\rho_{j} \to J \to 0$$
is also an exact sequence in $\Coh(\uX)$. Hence $J\in \Coh(C)\rho_{j}$, and by Lemma \ref{uch2} we have
$$
\mu_{\sigma}(J)=\mu_{\sigma}(J\otimes \OO_{\uX}(-(k-j)\sC))+\dfrac{(k-j)\sC^{2}}{H\sC}.
$$ 
Twisting by $\OO_{\uX}(-(k-j)\sC)$, we get an exact sequence in $\cA$:
$$
0 \to S\otimes \OO_{\uX}(-(k-j)\sC) \to F\rho_{k} \to J\otimes \OO_{\uX}(-(k-j)\sC) \to 0.
$$
Hence 
$\mu_{\sigma}(J\otimes \OO_{\uX}(-(k-j)\sC))\geq  \mu_{\sigma, \min}(F\rho_{k})$, and we have
$$
\mu_{\sigma, \min}(F\rho_{j})=\mu_{\sigma}(J)=\mu_{\sigma}(J\otimes \OO_{\uX}(-(k-j)\sC))+\dfrac{(k-j)\sC^{2}}{H\sC}
\geq \mu_{\sigma, \min}(F\rho_{k})+ \dfrac{(k-j)\sC^{2}}{H\sC}.
$$
\end{proof}

\begin{lemma}\label{nonzero}
	Assume $A, k$ are in case (i) of Lemma \ref{trouble1}, write $f_{*}(A)=\bigoplus_{j=0}^{n-1} A_{j}\rho_{j}$. Then for every $A_{j}\neq 0$, the induced morphism $A_{j} \to P_{k}$ in $D^{b}(X)$ is nonzero. 
	
	Similarly, assume $A', k$ are in case (a) of Lemma \ref{trouble1}, write $f_{*}(A')=\bigoplus_{j=0}^{n-1} A'_{j}\rho_{j}$. Then for every $A'_{j}\neq 0$, the induced morphism $T_{k} \to A'_{j} $ in $D^{b}(X)$ is nonzero.
\end{lemma}

\begin{proof}
We prove the first claim for $k=0$, other cases are similar. 
Assume for some $j$ the induced morphism $A_{j} \to P_{0}$ is zero. By Lemma \ref{zerocomponent}, we have either $A_{j}=0$ or $A_{j}$ is torsion free. By assumption of case (i), we have $A_{k}=0$. Hence $A$ is torsion and $A_{j}$ must be zero. 
\end{proof}

Now we prove the crucial technical proposition. 

\begin{proposition}\label{solution1}
	Fix any $t>0$ and assume $A$ is in case (i) of Lemma \ref{trouble1}. Then there is a constant $M_{7}$ that only depends on $\uX$ and $t$, such that 
	$$
	\mu_{\sigma}(f^{*}(P_{k}\rho_{k}))\leq \max(0, \mu_{\sigma}(E))+M_{7}.
	$$
	Similarly, assume $A'$ is in case (a) of Lemma \ref{trouble1}. Then there is a constant $M_{7}'$ that only depends on $\uX$ and $t$, such that 
	$$
	\mu_{\sigma}(f^{!}(P_{k}\rho_{k}))\geq \min (0, \mu_{\sigma}(E))-M_{7}'.
	$$
\end{proposition}

\begin{proof}
	We prove the proposition for $k=0$, the other cases are similar. Since $A\in \Coh(\uX)$ and $A_{0}=0$, we have $A\in D^{b}(X)^{\perp}\cap \Coh(\uX)$. By Lemma \ref{chain}, $A$ admits a filtration whose graded factors are 
	$$
	A_{n-1}\rho_{n-1}, \cdots, A_{1}\rho_{1}.
	$$
	Let $k$ be the minimal index so that $A_{k}\neq 0$. 
	Note that $A_{k}\rho_{k}$ is a quotient of $A$. Hence 
	\begin{equation}\label{AA}
		\mu_{\sigma, \min}(A_{k}\rho_{k}) \geq \mu_{\sigma}(A)\geq \mu_{\sigma}(P_{0}).
	\end{equation}
	\textbf{Step 1: Construct the destabilizing object.}	
	
	Recall that for every $j\in \ZZ$ and $F, F'\in \Coh(C)$, we have 
	$$\Ext^{1}_{\uX}(F\rho_{k}, F'\rho_{k+1})\cong \Hom_{C}(F, F'). $$
	We construct an object $A'$ whose graded factors are 
	$$ A_{n-1}\rho_{n-1}, \cdots, A_{k}\rho_{k}, A_{k}\rho_{k-1}, \cdots, A_{k}\rho_{1}, A_{k}\rho_{0}, $$
	by attaching $A_{k}\rho_{k-1}, \cdots, A_{k}\rho_{1}, A_{k}\rho_{0}$ to the previous objects in order, and the corresponding extension classes are the identity morphisms. Then we have a short exact sequence in both $\cA$ and $\Coh(\uX)$: 
	$$
	0 \to A \to A' \to A'' \to 0,
	$$
	where $A''$ admits a filtration whose factors are 
	\begin{equation}\label{a''}
		A_{k}\rho_{k-1}, \cdots, A_{k}\rho_{0}.
		\end{equation}
		 Note that $k-1<n-1$, hence by Lemma \ref{circle}, we have $\Hom_{\uX}(A'', P_{0})=\Ext^{1}_{\uX}(A'', P_{0})=0$. There is a short exact sequence
	$$
	0=\Hom_{\uX}(A'', P_{0}) \to \Hom_{\uX}(A', P_{0}) \to \Hom(A, P_{0}) \to \Ext^{1}_{\uX}(A'', P_{0})=0,
	$$
	hence $\Hom_{\uX}(A', P_{0}) \cong \Hom(A, P_{0})$. Let $\iota: A \hookrightarrow P_{0}$ be the inclusion of $A$ and $\iota'\in \Hom_{\uX}(A', P_{0})$ be the corresponding morphism. Then for $0\leq j \leq k$, the $\rho_{j}$ projection of $\iota'$ in $\Coh(X)$ is 
	$$ \iota'_{j}=\iota_{k}: A_{k} \to P_{0}. $$
	By Lemma \ref{nonzero}, we have $\iota'_{j}$ are all nonzero. In particular, $\iota'_{0}\neq 0$. Hence the composition morphism 
	$$
	A' \to P_{0} \to P
	$$
	is nonzero, because its $\rho_{0}$ component is the composition
	$
	A' \overset{\iota'_{0}}{\longrightarrow} P_{0} \overset{\id}{\longrightarrow} P. 
	$
	Since $A \to P$ is zero, we get a nonzero morphism 
	\begin{equation*}
		\gamma: A'' \too P.
	\end{equation*}
	
	\textbf{Step2: Obtain inequalities.}
 
	Since $\gamma: A'' \to P$ is nonzero, and $P$ is a subobject of the stable object $E$, by (\ref{a''}) we have
	\begin{equation}\label{J<}
		\min \left(\mu_{\sigma, \min}(A_{k}\rho_{k-1}), \cdots, \mu_{\sigma, \min}(A_{k}\rho_{0}) \right) \leq \mu_{\sigma, \min}(A'') \leq \mu_{\sigma}(E).
	\end{equation}
	Fix the $0\leq j \leq k-1$
	so that 
	$$
	\mu_{\sigma, \min}(A_{k}\rho_{j})=\min \left(\mu_{\sigma, \min}(A_{k}\rho_{k-1}), \cdots, \mu_{\sigma, \min}(A_{k}\rho_{0}) \right).
	$$
	 Let $J$ be a minimal $\sigma$-destablizing quotient of $A_{k}\rho_{j}$, and $S$ be its kernel. Then either $J\in \Coh(\uX)$ or not. 
	\\
	\textit{Case 1: $J\in \Coh(\uX)$:}
	Note that $A_{k}\rho_{j} \twoheadrightarrow J$ in both $\cA$ and $\Coh(\uX)$, by (\ref{J<}), (\ref{AA}), and Lemma \ref{minslope}, such that 
	\begin{equation}\label{torsionJ}
	\mu_{\sigma}(E)\geq	\mu_{\sigma, \min}(A_{k}\rho_{j}) \geq \mu_{\sigma, \min}(A_{k}\rho_{k})+\dfrac{(k-j)\sC^{2}}{H\sC} \geq \mu_{\sigma}(P_{0})+\dfrac{(k-j)\sC^{2}}{H\sC}. 
	\end{equation}
	\textit{Case 2: $J\not\in \Coh(\uX)$:} 
	In the long exact sequence
	$$
	0 \to \cH^{-1}(J) \to S \to A_{k}\rho_{j} \to \cH^{0}(J),
	$$
	we take the $\rho_{n-1}$ component. 
	Since $j\neq n-1$, we have $\cH^{-1}(J_{n-1}) \cong S_{n-1}$. By Proposition \ref{world1}, we have
	\begin{equation}\label{S<}
		-\alpha<\mu_{\min}^{B}(S)-\alpha\leq \mu_{\min}^{B}(S_{n-1})=\mu_{\min}^{B}(\cH^{-1}(J_{n-1})) \leq 
		 \mu_{\max}^{B}(\cH^{-1}(J))+ \alpha\leq \alpha. 
	\end{equation}
    Since $S_{\tor} \to A_{k}\rho_{j}$ is injective, we have $S_{\tor} \in \Coh(C)\rho_{j}$. Write $S_{\tor}=B \rho_{j}$. Taking quotient of $S_{\tor}$, we have a short exact sequence in $\cA$:
	\begin{equation*}
		0 \too S_{\mathrm{tf}} \too (A_{k}/B) \rho_{j}\too J \too 0. 
	\end{equation*}
	By (\ref{J<}), we have $\mu_{\sigma}(J)=\mu_{\sigma, \min}(A_{k}\rho_{j}) \leq \mu_{\sigma}(E)$. 
	By (\ref{S<}) and Lemma \ref{vertical}, we have $\mu_{\sigma}(S_{\mathrm{tf}}) \leq 0$. Since $J\not \in \Coh(\uX)$, we have $S_{\mathrm{tf}}\neq 0$, $J\neq 0$. Hence $(A_{k}/ R)\rho_{j}\neq 0$, and we have 
	\begin{equation}\label{A<}
		\mu_{\sigma}((A_{k}/R) \rho_{j}) \leq \max(\mu_{\sigma}(S_{\mathrm{tf}}), \mu_{\sigma}(J))\leq \max(0, \mu_{\sigma}(E)).
	\end{equation}
    By Lemma \ref{uch2} we have
	$$
	\mu_{\sigma}((A_{k}/R) \rho_{k})=\mu_{\sigma}((A_{k}/R) \rho_{j}\otimes \OO_{\uX}(j-k)\sC)=
	\mu_{\sigma}((A_{k}/R) \rho_{j})+\dfrac{(j-k)\sC^{2}}{H\sC}.
	$$
	Let $ M_{7}=\max_{0\leq a \leq b \leq n-1} \left(\dfrac{(a-b)\sC^{2}}{H\sC}\right) $.
	By (\ref{A<}) we have
	$$
 \mu_{\sigma}((A_{k}/R) \rho_{j})+\dfrac{(j-k)\sC^{2}}{H\sC}\leq \max(0, \mu_{\sigma}(E))+M_{7}.
    $$	
    Hence we have
    \begin{equation}\label{nontorsionJ}
    	\mu_{\sigma}(P_{0})\leq \mu_{\sigma}(A)\leq \mu_{\sigma, \min}(A_{k}\rho_{k})\leq 
    \mu_{\sigma}((A_{k}/R) \rho_{k}) \leq \max(0, \mu_{\sigma}(E))+M_{7}.
\end{equation}
Combining (\ref{torsionJ}) and (\ref{nontorsionJ}), the proposition is proved.
\end{proof}
Hence we have the following consequence.
\begin{proposition}\label{solutionP}
	Fix any $t>0$. Using Notation \ref{cases}, there is a constant $M_{8}$ that only depends on $\uX$ and $t$, such that for every $0\leq k \leq n-1$ we have
	$$
	\mu_{\sigma}(f^{*}(P_{k}\rho_{k})) \leq \max(0, \mu_{\sigma}(E))+M_{8}.
	$$
	Similarly, there is a constant $M_{8}'$ that only depends on $\uX$, such that 
	$$
	\mu_{\sigma}(f^{!}(T_{k}\rho_{k})) \geq \min(0, \mu_{\sigma}(E))- M_{8}'.
	$$
\end{proposition}

\begin{proof}
	The claim follows from Corollary \ref{solution2} and Proposition \ref{solution1}.
\end{proof}  

Now we can start to compute bounds.

\begin{lemma}\label{vertical}
	Fix $t>0$ and $ 0< a<\sqrt{2tH^{2}}$.
	\begin{itemize}
	\item Let $F $ be a torsion free sheaf with $0< \mu_{\min}^{B}(F)\leq \mu_{\max}^{B}(F)\leq a$. Then $\mu_{\sigma_{t}}(F)<0$. 
	\item Let $F'$ be a torsion free sheaf with 
	$- a \leq \mu_{\min}^{B}(F') \leq \mu_{\max}^{B}(F')\leq 0$. Then $\mu_{\sigma_{t}}(F'[1])>0$. 
	\end{itemize}
\end{lemma}

\begin{proof}
	We prove the first claim, the second claim is similar.
	Since $0< \mu_{\min}^{B}(F)\leq \mu_{\max}^{B}(F) \leq a$, by Proposition \ref{BOG} we have $\ch_{2}^{B}(F)\leq \dfrac{a^{2}}{2H^{2}}\ch_{0}^{B}(F)$. Hence 
	\begin{equation*}
		-\dfrac{\Re(Z(F))}{\Im(Z(F))}=\dfrac{\ch_{2}^{B}(F)-t\ch^{B}_{0}(F)}{H\ch^{B}_{1}(F)}\leq 
		\left(\dfrac{a^{2}}{2H^{2}}-t\right)\dfrac{\ch_{0}^{B}(F)}{H\ch_{1}^{B}(F)}
		<0.
	\end{equation*}
\end{proof}

\begin{proposition}\label{solutionch1}
     Fix any $t> 0$. 
     There is a constant $M_{9}$ that only depends on $\uX$ and $t$, such that for every $0\leq k \leq n-1$, we have
     $$\|\ch_{1}^{B}(E_{k})\| \leq M_{9}|Z(E)|. $$
\end{proposition}

\begin{proof}
	Fix any constant $0<a<\sqrt{2tH^{2}}$. 
	Note that $\cH^{-1}(E)_{>-a}[1] \hookrightarrow E$ and $E \twoheadrightarrow \cH^{0}(E)_{<a}$. By Lemma \ref{vertical}, we have 
	$$\mu_{\sigma}(\cH^{-1}(E)_{>-a}[1])>0, \mu_{\sigma}(\cH^{0}(E)_{<a})<0. $$
	 By stability of $E$, either $\cH^{-1}(E)_{>-a}[1]=0$ or $\cH^{0}(E)_{<a}=0$. We assume $\cH^{-1}(E)_{>-a}[1]=0$, the other case is proved similarly. We also assume $P \neq 0$, otherwise the claim follows from Proposition \ref{Q}.
	
	By Proposition \ref{Q}, it suffices to find a constant $b$ such that $\|\ch_{1}(P_{k})\|\leq b|Z(E)|. $
	Since $f^{*}(P_{k}\rho_{k}) \to P$ is a generic isomorphism, we have $\ch_{1}(f^{*}(P_{k}\rho_{k}))=\ch_{1}(P)+n\sC$ for some $n\in \ZZ$. By Proposition \ref{Q} and Theorem \ref{supp}, there is a constant $b_{4}$ such that $\|\ch_{1}(P)\|\leq b_{4}|Z(E)|$. Hence it suffices to bound $n$.
	By Proposition \ref{world1}, we have 
	\begin{align*}
		H\ch_{1}^{B}(f^{*}(P_{k}\rho_{k}))=&H\ch_{1}^{B}(\cH^{0})(f^{*}(P_{k}\rho_{k})) -H\ch_{1}^{B}(\cH^{-1}(f^{*}(P_{k}\rho_{k})))\\
		\leq & H\ch_{1}^{B}(\cH^{0}(P))+H\ch_{1}^{B}(\cH^{-1}(P))-\alpha\ch_{0}^{B}(\cH^{-1}(P))\\
		= & H\ch_{1}^{B}(P)-\alpha\ch_{0}^{B}(\cH^{-1}(P)).
	\end{align*}
	Since $\cH^{-1}(E)_{>-a}[1]=0$, we have $\mu_{H}(\cH^{-1}(P))\leq -a$. Hence 
	$$
	\dfrac{H\ch_{1}^{B}(f^{*}(P_{k}\rho_{k}))}{H\ch_{1}^{B}(P)}
	\leq \dfrac{H\ch_{1}^{B}(P)-\alpha\ch_{0}^{B}(\cH^{-1}(P))}{H\ch_{1}^{B}(P)}
	= 1-\alpha\dfrac{1}{\mu^{B}(\cH^{-1}(P))}\leq 1+\dfrac{\alpha}{a}. 
	$$
	Therefore $n\leq \dfrac{\alpha}{aH\sC}H\ch_{1}^{B}(P) \leq \dfrac{\alpha}{aH\sC}|Z(E)|$. Hence we showed that there is a constant $b_{5}$, such that $\|\ch_{1}^{B}(f^{*}(P_{k}\rho_{k}))\| \leq b_{5}|Z(E)|$. 
	Note that 
	$$
	\ch_{1}^{B}(f^{*}(P_{k}\rho_{k}))=\ch_{1}^{B}(P_{k})-k\sC\ch_{0}^{B}(P).
	$$
	By Proposition \ref{Q} and Theorem \ref{supp}, there is a constant $b_{6}$, such that $\|\ch_{0}^{B}(P)\|\leq b_{6}|Z(E)|$. The proposition is proved. 
\end{proof}

\begin{proposition}\label{solutionch2}
	Fix any $t> 0$. 
	 There is a constant $M_{10}$ that only depends on $\uX$ and $t$, such that for every $0\leq k \leq n-1$, we have
	$$ \|\ch_{2}^{B}(E_{k})\| \leq M_{10}|Z(E)|. $$
\end{proposition}

\begin{proof}
	We first prove that there is a constant $b$, such that for every $0\leq k \leq n-1$, we have
	$$
	\ch_{2}^{B}(P_{k})\leq b|Z(E)|. 
	$$
	We use Notation \ref{cases}.
	Fix any constant $0<a<\sqrt{2tH^{2}}$. 
	Note that $\cH^{-1}(E)_{>-a}[1] \hookrightarrow E$ and $E \twoheadrightarrow \cH^{0}(E)_{<a}$. By Lemma \ref{vertical}, we have 
	$$\mu_{\sigma}(\cH^{-1}(E)_{>-a}[1])>0, \mu_{\sigma}(\cH^{0}(E)_{<a})<0. $$
	By stability of $E$, either $\cH^{-1}(E)_{>-a}[1]=0$ or $\cH^{0}(E)_{<a}=0$. We assume $\cH^{-1}(E)_{>-a}[1]=0$, the other case is proved similarly.
	 We also assume $P\neq 0$, otherwise the claim follows from Proposition \ref{Q}.
	
	Note that $|\Re(Z(E))| \leq |Z(E)|$. By Proposition \ref{solutionP}, we have 
	$$
	\mu_{\sigma}(f^{*}(P_{k}\rho_{k}))=
	\dfrac{\ch_{2}^{B}(f^{*}(P_{k}\rho_{k}))-t\ch_{0}^{B}(f^{*}(P_{k}\rho_{k}))}{H\ch_{1}^{B}(f^{*}(P_{k}\rho_{k}))}
	\leq -\dfrac{\Re(Z(E))}{H\ch_{1}^{B}(E) }\leq \dfrac{|Z(E)|}{H\ch_{1}^{B}(E) }.
	$$
	Hence 
	$$
	\ch_{2}^{B}(f^{*}(P_{k}\rho_{k}))\leq \dfrac{H\ch_{1}^{B}(f^{*}(P_{k}\rho_{k}))}{H\ch_{1}^{B}(E)}|Z(E)|+t\ch_{0}^{B}(P).
	$$
    By Proposition \ref{world1}, we have 
    \begin{align*}
H\ch_{1}^{B}(f^{*}(P_{k}\rho_{k}))=&H\ch_{1}^{B}(\cH^{0}(f^{*}(P_{k}\rho_{k}))) -H\ch_{1}^{B}(\cH^{-1}(f^{*}(P_{k}\rho_{k})))\\
 \leq & H\ch_{1}^{B}(\cH^{0}(P))+H\ch_{1}^{B}(\cH^{-1}(P))-\alpha\ch_{0}^{B}(\cH^{-1}(P))\\
 = & H\ch_{1}^{B}(P)-\alpha\ch_{0}^{B}(\cH^{-1}(P)).
    \end{align*}
    Since $\cH^{-1}(E)_{>-a}[1] =0$, we have $\mu^{B}(\cH^{-1}(P))\leq -a$. Hence 
    $$
    \dfrac{H\ch_{1}^{B}(f^{*}(P_{k}\rho_{k}))}{H\ch_{1}^{B}(E)}
    \leq \dfrac{H\ch_{1}^{B}(P)-\alpha\ch_{0}^{B}(\cH^{-1}(P))}{H\ch_{1}^{B}(E)}
    \leq 1-\alpha\dfrac{1}{\mu^{B}(\cH^{-1}(P))}\leq 1+\dfrac{\alpha}{a}. 
    $$
    By Proposition \ref{Q} and Theorem \ref{supp}, there is a constant $a_{1}$, such that $\ch_{0}^{B}(P)\leq a_{1}|Z(E)|$. Hence 
    $$
    \ch_{2}^{B}(f^{*}(P_{k}\rho_{k}))\leq (1+\dfrac{\alpha}{a}) |Z(E)|+ta_{1}|Z(E)|=(1+\dfrac{\alpha}{a} +a_{1})|Z(E)|.
    $$
    Since $f^{*}(P_{k}\rho_{k})=P_{k}(-k\sC)$, we have
    $$
    \ch_{2}^{B}(f^{*}(P_{k}\rho_{k}))=\ch_{2}^{B}(P_{k})-k\ch_{1}^{B}(P_{k})\sC+\frac{k^{2}\ch_{0}^{B}(P_{k})\sC^{2}}{2}.
    $$
    By Proposition \ref{Q} and Theorem \ref{supp}, there are constants $a_{2}, a_{3}$, such that 
    $$\ch_{1}^{B}(P_{k})\sC \leq a_{2}|Z(E)|,\quad |\ch_{0}^{B}(P_{k})|\leq a_{3}|Z(E)|. $$
    Hence there is a constant $b$, such that 
    \begin{equation}\label{less}
    \ch_{2}^{B}(P_{k})\leq b |Z(E)|. 
    \end{equation}
    By Lemma \ref{chlink}, we have 
    $$
    \ch_{2}^{B}(P_{k})=\ch_{2}^{B}(P)+\sum_{j\neq k}(\ch_{2}^{B}(P)-\ch_{2}^{B}(P_{j}))-\sum_{j=0}^{n-1}\dfrac{n-2j-1}{2n^{2}}\ch_{1}^{B}(P_{k})C.
    $$
    By Proposition \ref{Q}, there are constants $a_{4}, a_{5}$ such that $\|\ch_{2}^{B}(P)\|\leq a_{4}|Z(E)|$, $\|\ch_{1}^{B}(P_{k})\| \leq a_{5}\| \ch_{1}^{B}(E_{k}) \|$. By Proposition \ref{solutionch1}, $\|\ch_{1}^{B}(E_{k})\|\leq M_{9} |Z(E)|$. Hence there is a constant $a_{6}$, such that 
    \begin{equation}\label{more}
    \ch_{2}^{B}(P_{k})\geq -a_{4}n|Z(E)|-(n-1)a_{5}M_{9}|Z(E)|- \sum_{j=0}^{n-1}\dfrac{n-2j-1}{2n^{2}}a_{6}|Z(E)|. 
\end{equation}
Combining (\ref{less}) and (\ref{more}), there is a constant $a_{7}$ such that $\|\ch_{2}^{B}(P_{k})\|\leq a_{7}|Z(E)|$. The proposition now follows from Proposition \ref{Q}.
\end{proof}

\begin{proof}[Proof of Proposition \ref{mainP}]
This follows by Proposition \ref{solutionch1} and Proposition \ref{solutionch2}.    
\end{proof}

\subsection{Support property on $\HH^{even}_{CR}(\uX,\mathbb{Q})$}
Now we prove the main theorem of this paper.
\begin{proof}[Proof of Theorem \ref{maintheorem}]
By Definition \ref{lattice}, we need to show the support property with respect to $\Lambda$. 
By Lemma \ref{normequivalence}, it suffices to bound $\| E \|_{B}$ for every $\sigma$-semistable $E\in \cA$.	By Theorem \ref{supp}, there is a constant $M_{11}$ such that $\|\ch^{B}(E)\|\leq M_{11} |Z(E)|$. By Proposition \ref{link}, we have $\ch^{B}(G_{0, k})=\ch^{B}(E_{k})-\ch^{B}(E_{0})$. By Proposition \ref{solutionP} we have $\|\ch^{B}(E_{j})\|\leq M_{E}|Z(E)|$ for every $0\leq j \leq n-1$. Hence 
	$$
	\|\ch^{B}(G_{0,k})\|=\|\ch^{B}(E_{k})-\ch^{B}(E_{0})\|\leq \|\ch^{B}(E_{k})\| + \|\ch^{B}(E_{0})\|
	\leq 2M_{E} |Z(E)|.
	$$
\end{proof}

Now we deform the stability conditions to the irrational case.

\begin{theorem}
Let $B, H\in \NS(\uX)_{\RR}$ and $H$ be an ample class. Then the tilt stability condition (Definition \ref{tilt stability}) is a Bridgeland stability that satisfies the support property with respect to $\HH^{even}_{CR}(\uX, \QQ)$.
\end{theorem}

\begin{proof}
   By Lemma \ref{prestability condition} and Theorem \ref{supp}, for $B, H \in \NS(X)_{\mathbb{Q}}$, $\sigma_{B, H}=(Z_{B,H}, \cA_{B,H})$ is a Bridgeland stability condition. To deform $B, H$ to $\NS(X)_{\RR}$, the proof follows from  Theorem \ref{deformation} and the following lemma, which can be proven exactly as \cite[Lemma 6.20]{MS17}. 
\end{proof}

\begin{lemma}[\cite{MS17}]
    Let $(Z_{B,H}, \cA)$ be a stability condition for which all skyscraper sheaves $\mathcal{O}_{p}\rho_{k}$ for $p\in C, 0\leq k \leq n-1$ and $\mathcal{O}_{q}$ for $q \not \in C$ are stable of phase one. Then $\cA=\cA_{B, H}$.
\end{lemma}
We recall the notion of twisted stability.
Suppose $E$ is a torsion free sheaf on $\uX$, let $\nu_{B,H}(E):=\ch_{2}^{B}(E)/\operatorname{rk}(E)$.
\begin{definition}
    We say $E$ is $(B,H)$-twisted semistable if for every proper nonzero subsheaf $F\subset E$, we have
    $$ \mu^{B}(F)< \mu^{B}(E) \ \operatorname{or} \ ( \mu^{B}(F)=\mu^{B}(E) \ \operatorname{and} \  \nu_{B,H}(F) \leq \nu_{B, H}(E)). $$
\end{definition}
We have the following large volume limit theorem, similar to that on a surface.
\begin{theorem}[Large volume limit]\label{largevolumelimit}
    Let $B, H \in \NS(X)_{\RR}$ with $H$ ample. For $t>0$, let $ \sigma_{t}:=(Z_{B,tH}, \cA_{\beta,tH}) $. Suppose $E\in D^{b}(\uX)$ satisfies $r(E), \mu^{B}(E)> 0$. Then $E$ is semistable under $\sigma_{t}$ for all $t\gg 0$ precisely if $E$ is a shift of a $(B,H)$-twisted semistable sheaf on $\uX$.
\end{theorem}

\begin{proof}
    The proof is the same as \cite[Proposition 14.2]{Bri08}.
\end{proof}

We end the paper by the following example.
\begin{example}[Deforming to general directions]
By Theorem \ref{maintheorem} and Bridgeland's deformation result (Theorem \ref{deformation}), we obtain Bridgeland stability conditions that are not tilts. For instance we consider the following special classes of deformations of $Z_{t}$. For every $(n-1)$-tuple $\beps=(\epsilon_{1}, \cdots, \epsilon_{n-1})\in \RR^{n-1}$ and $\beps'=(\epsilon'_{1}, \cdots, \epsilon'_{n-1})\in \RR^{n-1}$, we define the following stability functions:
\begin{equation*}
	Z_{B,H, t, \beps, \beps'}(\vv)=-\left(\ch_{2}^{B}(\vv)+\sum_{k=1}^{n-1} \epsilon_{k}\ch_{2}(\mathrm{pr}_{k}(\vv))\right) + t\ch_{0}^{B}(\vv)
	+ i H \cdot \left(\ch_{1}^{B}(\vv)+ \sum_{k=1}^{n-1}\epsilon'_{k}\ch_{1}^{B}(\mathrm{pr}_{k}(\vv))\right). 
\end{equation*}

	For $\|\beps\|, \|\beps' \| \ll 1$, there exists the heart of a bounded $t$-structure $\cA_{B,H, t, \beps, \beps'}$, such that the pair 
 $$
 \sigma_{B,H, t, \beps, \beps'}=(Z_{B,H, t, \beps, \beps'}, \cA_{B,H, t, \beps, \beps'})
 $$ 
 is a Bridgeland stability condition. When $\beps\neq 0$ or $\beps'\neq 0$, $\sigma_{B,H, t, \beps, \beps'}$ is not a tilt stability condition (Definition \ref{tilt stability}). Note that for $p\in C$, the skyscraper sheaves $\OO_{p}$ can be unstable.

 To obtain the optimal bounds for $\beps, \beps'$ is an interesting question. 
\end{example}

\section*{References}

\bibliographystyle{alpha}
\renewcommand{\section}[2]{} %gets rid of 'REFERENCES'
\bibliography{ref}

\end{document}